\RequirePackage{fix-cm}
\documentclass[smallextended]{svjour3a}       
\smartqed  
\usepackage[T1]{fontenc}
\usepackage[utf8]{inputenc}
\usepackage{color}
\usepackage{array,epsfig,amsfonts,mathrsfs,amsmath,cite}
\usepackage{amstext}
\usepackage{graphicx}
\usepackage{setspace}
\usepackage{esint}
\makeatletter
\pdfpageheight\paperheight
\pdfpagewidth\paperwidth

\@ifundefined{definecolor}
{\usepackage{color}}{}


\allowdisplaybreaks

\newcommand{\R}{\mathbb{R}}

\newcommand{\abs}[1]{\left\vert#1\right\vert}

\newcommand{\bracket}[1]{\left[#1\right]}

\newcommand{\di}{{\rm div}}

\DeclareSymbolFont{lettersA}{U}{pxmia}{m}{it}
\DeclareMathSymbol{\piup}{\mathord}{lettersA}{"19}

\newcommand{\bn}{\mathbf n}
\newcommand{\bO}{O}
\newcommand{\bu}{\mathbf u}
\newcommand{\bv}{\mathbf v}

\newcommand{\bx}{\mathbf x}
\newcommand{\by}{\mathbf y}

\newcommand{\bk}{\mathbf k}

\newcommand{\bnu}{\boldsymbol\nu}
\newcommand{\btau}{\boldsymbol\tau}

\newcommand{\bmu}{\boldsymbol\mu}

\newcommand{\cd}{\mathcal{D}}
\newcommand{\cf}{\mathcal{F}}

\newcommand{\cs}{\mathcal{S}}
\newcommand{\ct}{\mathcal{T}}

\newcommand{\cw}{\mathcal{W}}
\newcommand{\cE}{\mathcal{E}}
\newcommand{\cq}{\mathcal{Q}}

\newcommand{\ft}{\mathscr{T}}

\newcommand{\tvf}{\tilde{\varphi}}

\newcommand{\ga}{\alpha}
\newcommand{\gb}{\beta}

\newcommand{\gd}{\delta}
\newcommand{\gD}{\Delta}
\newcommand{\gO}{\Omega}
\newcommand{\vf}{\varphi}
\newcommand{\ve}{\varepsilon}

\newcommand{\ls}{\bar{s}}
\newcommand{\lw}{\bar{w}}
\newcommand{\ts}{\tilde{s}}
\newcommand{\hs}{\hat{s}}
\newcommand{\DD}{D}
\makeatother

\makeatother

\usepackage{babel}
  \providecommand{\definitionname}{Definition}
  \providecommand{\lemmaname}{Lemma}
  \providecommand{\propositionname}{Proposition}
  \providecommand{\remarkname}{Remark}
\providecommand{\theoremname}{Theorem}
\providecommand{\problemname}{Problem}

\numberwithin{equation}{section}
\numberwithin{figure}{section}
\newtheorem{thm}{\protect\theoremname}
  \newtheorem{defn}{\protect\definitionname}
  \newtheorem{prop}{\protect\propositionname}
  \newtheorem{rem}{\protect\remarkname}
  \newtheorem{lem}{\protect\lemmaname}
\newtheorem{prob}{\protect\problemname}

\counterwithin{figure}{section}
\numberwithin{equation}{section}
\numberwithin{thm}{section}
\numberwithin{lem}{section}
\numberwithin{prop}{section}
\numberwithin{corollary}{section}
\numberwithin{defn}{section}
\numberwithin{rem}{section}
\numberwithin{prob}{section}

\journalname{Communications in Mathematical Physics (2021),}
\begin{document}

\title{Stability of Attached Transonic Shocks in Steady Potential Flow past Three-Dimensional Wedges}
\titlerunning{Stability of Three-Dimensional Attached Transonic Shocks}

\author{Gui-Qiang G. Chen\, \and Jun Chen\,  \and Wei Xiang}


\institute{Gui-Qiang G. Chen \at
          Mathematical Institute,\  University of Oxford, Oxford, OX2 6GG, UK\\
            \email{chengq@maths.ox.ac.uk}           
         \and
          Jun Chen \at
          Center of Applied Mathematics, Yichun University, Yichun, Jiangxi 336000, China\\
              \email{chenjun@jxycu.edu.cn}           
           \and
     Wei Xiang \at
     Department of Mathematics,	City University of Hong Kong, Hong Kong, China\\
              \email{weixiang@cityu.edu.hk}
}

\date{Received: May 8, 2020 / Accepted: June 29, 2021}
\maketitle

\begin{abstract}
$\,$ We develop a new approach and employ it to establish the global existence and nonlinear structural stability of
attached weak transonic shocks in steady potential flow past three-dimensional wedges; in particular,
the restriction that the perturbations are away from the wedge edge in the previous results is removed.
One of the key ingredients is to identify a {\it good} direction of the boundary operator of a boundary condition
of the shock along the wedge edge, based on the non-obliqueness of the boundary condition for the weak shock on the edge.
With the identification of this direction, an additional boundary condition on the wedge edge can be assigned to make sure
that the shock is attached on the edge and linearly stable under small perturbations.
Based on the linear stability, we introduce an iteration scheme and prove that there exists a unique fixed point of
the iteration scheme, which leads to the global existence and nonlinear structural stability of the attached weak transonic shock.
This approach is based on neither the hodograph transformation nor the spectrum analysis,
and should be useful for other problems with similar difficulties.


\end{abstract}

\section{$\,$ Introduction}
We are concerned with the stability of attached transonic shocks in steady flow past three-dimensional wedges
with non-flat surfaces.
This is a longstanding problem
at least dating back to Prandtl \cite{Prandtl} in 1936,
in which it was first conjectured that the weak shock solution is stable.
In this paper,
we develop a new approach and employ it to establish
the global existence and nonlinear
structural stability of attached weak transonic shocks in steady flow
past three-dimensional wedges with non-flat surfaces, governed by the three-dimensional Euler
equations for potential flow.
The perturbations of the wedge and the incoming flow are allowed up to the wedge edge,
which removes the assumption in \cite{CF} that the perturbations are away from the wedge edge.

As indicated in \cite{CoF}, when a uniform supersonic flow passes a symmetric wedge with flat surface
whose (half) wedge-angle is less than the critical angle,
an attached plane shock is expected to be generated.
There are two solutions satisfying the physical entropy condition.
The solution with smaller density of the constant downstream flow
is called a weak shock solution,
while the other is called a strong shock solution.
The downstream state of the strong shock solution is always subsonic,
but the downstream state of the weak shock solution can be either subsonic or supersonic,
depending on the wedge-angle.
Thus, a natural question is which one, or both, could be actually physical.
There has been a long debate about whether the strong shock or the weak shock, or both, would be stable
starting 1930s (see Prandtl \cite{Prandtl}, Courant-Friedrichs \cite[Section 123]{CoF},
and von Neumann \cite{Neumann}; also see \cite{Liu,Serre});
this is partly because it is a basic principle
in physics that a physical shock solution
should be stable under small perturbations in an appropriate sense.
Therefore, it is important to study the stability or instability of such plane
shock solutions to single out the physical ones.

Some satisfactory results for the two-dimensional case have been obtained.
In particular, the existence, uniqueness, stability, and asymptotic behavior of solutions
under a small perturbation for both the strong and weak shocks have been established.
We refer the reader
to \cite{CCF,CCF1,ChenLi,ChenZhangZhu,Chen2,Chen3,Sxchen,ChenFang,Fang,FX,Gu,Schaeffer,YinZhou2009JDE,Zh1,Zh2}
for more details; also see Chen-Feldman \cite{CF2}.
On the other hand, the three-dimensional case is completely different.
For the potential flow, Li-Xu-Yin \cite{LXY} showed that the three-dimensional attached strong shock
past a sharp wedge is not stable with respect to a periodic perturbation,
while Chen-Fang \cite{CF} proved that the three-dimensional weak shock is stable
if the perturbations are away from the wedge edge.

In this paper, we remove the restriction in \cite{CF} and prove that the three-dimensional weak shock is stable
even if both perturbations of the wedge edge and the incoming flow up to the wedge edge are allowed.
This provides an answer to the issue that has been debated since Courant-Friedrichs \cite{CoF}
and von Neumann \cite{Neumann} for the stability
of the attached transonic shock governed by the potential flow equation.

To achieve this, we
first formulate the shock problem as a free boundary problem
and then develop a different approach to handle the free boundary problem
from the ones used in \cite{CF,LXY}
by identifying
a {\it good} direction of the boundary operator of a boundary condition
on the shock along the wedge edge,
based on the non-obliqueness of the boundary condition for the weak shock on the edge.
The identification of this {\it good} direction allows us to assign an additional boundary condition on the wedge edge to make
sure that the shock is attached on
the edge under the small perturbations.
Based on this observation, we design a barrier function of the solutions near the wedge edge,
which allows us to show the $C^{1,\alpha}$--regularity of solutions near the edge.
Then we can establish the linear stability by constructing a solution of the linearized problem
via cutting off the wedge edge and passing to the limit.
Based on the linear stability, we adapt an iteration scheme and prove that there exists
a unique fixed point of the iteration scheme, which leads to the global existence and
nonlinear structural stability of the attached weak transonic shock. This approach
is based on neither the partial hodograph transformation nor the spectrum analysis, and
should be useful for other problems with similar difficulties.

There are other related references on various free boundaries for the Euler equations,
such as \cite{BCF,BCF1,BCF2,CF1,CF2,CFX,CY,EL} for the global self-similar shock solutions past wedges
and \cite{CHWX,CDX,CDX1,CDX2,HKWX,QW,XZZ} for the Euler flows without shocks of strong strength.

The rest of the paper is organized as follows:
In \S 2, we formulate the shock problem as a free boundary problem (Problem 2.1) and then state the main theorem of the paper (Theorem 2.1).
In \S 3, we reformulate the free boundary problem into a nonlinear fixed boundary problem by introducing the coordinate transformation
and then introduce the iteration scheme in the new coordinates.
In \S 4, we show that the iteration scheme is well-defined by proving that the linearized problem introduced in \S 3 can be uniquely solved.
Finally, we establish the main theorem (Theorem 2.1) in \S 5 by applying the Banach fixed-point theorem.

\section{$\,$ Mathematical Formulation and Main Theorem -- Theorem 2.1}
In this section, we formulate the shock problem as a free boundary problem and then state
the main theorem of this paper, Theorem 2.1.

\subsection{$\,$ Wedge Surfaces and the Euler Equations for Potential Flow}
Define the wedge surface by
\begin{align}\label{def-wedge}
\cw = \big\{\bx\in \R^3\,:\,(x_1,x_3) \in \cd_{e_1} , x_2 = w(x_1,x_3)\big\},
\end{align}
where
\begin{align}\label{def-wedgedomain}
\cd_{e_1} = \big\{(x_1,x_3) \in \R^2\,:\,x_1 > e_1(x_3) \big\}.
\end{align}
Set
\begin{align*}
e_2(x_3) = w( e_1(x_3) , x_3).
\end{align*}
Then
\begin{align}\label{def-edge}
\cE = \big\{\bx=(x_1,x_2 ,x_3)\in \R^3\,:\, x_1 = e_1(x_3), x_2 = e_2(x_3), x_3 \in \R \big\}
\end{align}
is the wedge edge.

In this paper, we focus on the compressible flows past over the wedge surface $\cw$
governed by the Euler equations for potential flow:
\begin{equation}
\di\big(\rho(\abs{\DD\varphi}^{2})\DD\varphi\big)=0,  \qquad\,\, \bx\in \R^3, \label{eq:potential}
\end{equation}
where $\varphi=\varphi(\bx)$ is the velocity potential so that $\bu=\DD\varphi$,
$q=|\DD\varphi|=|\bu|$ is the speed, and
$\rho$ is the density with
\begin{align}\label{def-rho}
\rho(q^{2})=\big(1-\frac{\gamma-1}{2}q^{2}\big)^{\frac{1}{\gamma-1}}
\end{align}
by scaling from the Bernoulli law for polytropic gases with adiabatic
exponent $\gamma>1$, and the gradient $\DD:=(\partial_{x_1}, \partial_{x_2}, \partial_{x_3})$.

The sonic speed is
\begin{align*}
c(q^{2})=\big(1-\frac{\gamma-1}{2}q^{2}\big)^{\frac{1}{2}}.
\end{align*}
The flow is called supersonic if
\begin{align}
|\DD\vf| > c(|\DD \vf|^2) \label{def-supersonic}
\end{align}
and called subsonic if
\begin{align}
|\DD\vf| < c(|\DD \vf|^2). \label{def-subsonic}
\end{align}
By the definition of $\rho$ in \eqref{def-rho}, supersonicity \eqref{def-supersonic} is equivalent to
\begin{align*}
|\DD\vf| > c_*,
\end{align*}
and subsonicity \eqref{def-subsonic} is equivalent to
\begin{align*}
|\DD\vf| < c_*,
\end{align*}
where $c_*=\frac{1}{\sqrt{\gamma +1}}$.

Moreover, the potential flow equation \eqref{eq:potential} can be written
in the following non-divergence form:
\begin{align}\label{eqn-nondiv}
\sum_{i,j=1}^{3}a_{ij}(\DD\varphi)\partial_{x_{i}x_{j}}\varphi=0,
\end{align}
where
\begin{align}
a_{ij}(\DD\varphi)&=
c^{2}(\abs{\DD\varphi}^{2})\gd_{ij}
-\partial_{x_{i}}\varphi\partial_{x_{j}}\varphi  \label{exp-aij}
\end{align}
and
$\gd_{ij}=1$ if $i=j$
and $0$ if $i\not=j$.

\subsection{$\,$ Shock Solutions}
Let $\cs$ be a $C^1$--surface separating an open domain $\gO$ into $\gO^-$ and $\gO^+$.
Let $\vf$ be both a $C^1$--function in each subdomain of $\gO^\pm$
and a weak solution in $W^{1,\infty}(\Omega)$ of \eqref{eq:potential} in the distributional sense in $\Omega$.
Denote
\begin{align*}
\vf^+ := \vf |_{\overline{\gO^+}}, \qquad \vf^- := \vf |_{\overline{\gO^-}}.
\end{align*}
Then $\varphi^{+}$ and $\varphi^-$ are classic solutions of \eqref{eq:potential}
in $\gO^+$ and $\gO^-$, respectively.
We can see from integration by parts that $\vf$ must satisfy
the following Rankine-Hugoniot conditions on  $\cs$:
\begin{align}
& \vf^+ = \vf^- ,\label{con-RH1}\\
& \big[\rho(\abs{\DD\varphi}^{2})\DD\varphi\big]\cdot \bnu =0,\label{con-RH2}
\end{align}
where $\bnu$ is the unit normal on $\cs$, and the bracket of a function denotes the difference between
the limiting values ({\it i.e.}, traces) of the function from both sides on $\cs$.
Differentiating \eqref{con-RH1} along the tangential direction $\btau$ on $\cs$ leads to
\begin{align}
\DD \vf^+ \cdot \btau  = \DD \vf^- \cdot \btau  \qquad \mbox{on $\cs$}.  \label{con-RH4}
\end{align}
Therefore,
\begin{equation}\label{2.4a}
\mbox{$\bracket{\DD \vf }\qquad $ is parallel to the normal $\bnu$ on $\cs$.}
\end{equation}

Furthermore, surface $\cs$ is called a shock provided that $\vf$ satisfies
the additional physical entropy condition on $\cs$:
\begin{align}
\rho(\abs{\DD\varphi^-}^{2}) <  \rho(\abs{\DD\varphi^+}^{2}), \label{con-entropy}
\end{align}
if the flow direction is from $\gO^-$ to $\gO^+$.
In this case, the corresponding piecewise $C^1$--function $\varphi$
is called a shock solution in $\Omega$; also see \cite{CF2,CoF,Dafermos}.

Set
\begin{align}\label{def-H}
H (\bu, \bv) := \big(\rho(|\bu|^2)\bu - \rho(|\bv|^2)\bv\big)\cdot (\bu-\bv) \qquad \mbox{for $\bu,\bv \in \R^3$}.
\end{align}

It is direct to see that, by \eqref{2.4a}, condition \eqref{con-RH2} is equivalent to
\begin{align}
H(\DD\vf^-, \DD \vf^+)
 =0  \qquad \mbox{ on $\cs$}.       \label{con-RH3}
\end{align}

To study the stability of the attached planar shock,
we regard the planar shock as a background solution.

\begin{defn}[Background Solutions]
$\,$ A piecewise constant function $\vf$ defined in $\Omega$
is called a {\it background solution} if $\vf$ satisfies the following conditions{\rm :}
\begin{enumerate}
\item  The plane shock $\cs$ divides $\Omega$  into $\gO^-$ and $\gO^+${\rm ;}
\item In both subdomains  $\gO^-$ and $\gO^+$, $\DD \vf^-$ and $\DD \vf^+ $ are constant vectors, respectively{\rm ;}
\item  $\vf^-$ is supersonic and $\vf^+$ is subsonic{\rm ;}
\item  $ \vf^-$ and $\vf^+ $ satisfy the Rankine-Hugoniot conditions \eqref{con-RH1}--\eqref{con-RH2} and the entropy condition \eqref{con-entropy}.
\end{enumerate}
\end{defn}

In this paper, we are concerned with the stability of the weak transonic shock background solution generated by a constant supersonic flow onto a flat wedge
with (half) wedge-angle $\theta_{\rm w}$.
It is direct from conditions \eqref{con-RH4}--\eqref{con-RH3} to find a piecewise constant transonic flow in the following way:
Define
\[
 \DD \vf^- := U^- =(u_1^-,u_2^-, u_3^-), \qquad \DD \vf^+ := U^+ =(u_1^+,u_2^+, u_3^+).
\]
The three-dimensional space for the phase states $(u_1,u_2,u_3)$ is our phase space.
Let
\[
U^-_0 = (q_0,0,0), \quad q_0 > c_*.
\]
Let $\theta_{\rm i}$
be the angle between $U_0^-$ and the wedge edge.
In the $\mathbf{x}$--coordinate system, if the flat wedge is symmetric with respect to the $(x_1, x_3)$--plane
and the wedge edge passes the origin, then the plane wedge function is
$$
w(x_1,x_3)=(x_1\sin\theta_{\rm i}-x_3\cos\theta_{\rm i})\tan\theta_{\rm w} \qquad \mbox{defined on $\{x_1>x_3\cot\theta_{\rm i}\}$};
$$
that is, $e_1(x_3)=x_3\cot\theta_{\rm i}$.

In the phase space, all the possible downstream velocity states $U=(u_1,u_2,u_3)$ of piecewise constant transonic flows
that connect with $U^-_0$ by plane shocks together form a balloon.

For the convenience of computation of the shock wave,
we rotate the coordinate system with a proper angle such that the plane wedge surface is the $(x_1,x_3)$--plane
(\emph{i.e.}, the downstream subsonic flow is parallel to the $(x_1,x_3)$--plane), and the $x_3$--axis lies
in the shock plane (see Fig. \ref{fig:rotation}).
Now, in the new coordinates, $e_1(x_3)=0$ and $w(x_1,x_3)=0$.

\begin{figure}
	\centering
	\includegraphics[width=0.7\textwidth]{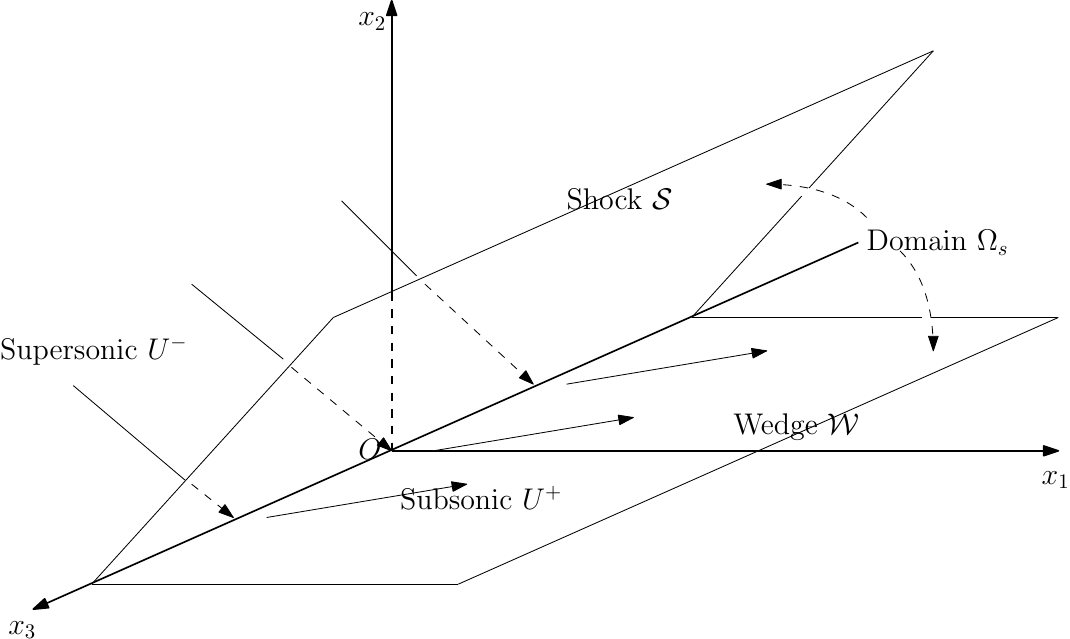}
	\protect\caption{The new coordinates after the rotation.\label{fig:rotation}}
\end{figure}

Without loss of the generality, we analyze the free boundary problem in the new coordinates, which are still denoted by $(x_1,x_2,x_3)$.
Then, in the coordinates, we know that $u_2^+ =0$ and, by condition  \eqref{con-RH4}, $u_3^+ =u_3^-$.
Then we have
\begin{align}
 U_0^- &=(q_0 \sin \theta_{\rm i} \cos \theta_{\rm w} ,-q_0 \sin \theta_{\rm i} \sin \theta_{\rm w}, q_0\cos \theta_{\rm i}), \label{def-u0-}\\
  U_0^+ &=(u_0 ,0, q_0\cos \theta_{\rm i}).\label{def-u0+}
\end{align}

Notice that the wedge-angle $\theta_{\rm w}$ and the angle $\theta_{\rm i}$
between $U_0^-$ and the wedge edge
are independent of the change of the coordinates.

For the given weak transonic shock solution $(U_0^-, U^+_0)$ defined in \eqref{def-u0-}--\eqref{def-u0+},
the shock surface is
\begin{align*}
\cs_0 = \{\bx \in \R^3\, :\, x_2 = \sigma x_1, x_1 \ge 0\}
\end{align*}
with
\begin{align}\label{exp-k0}
\sigma = \frac{q_0 \sin \theta_{\rm i} \cos \theta_{\rm w} - u_0}{q_0 \sin\theta_{\rm i} \sin \theta_{\rm w}}.
\end{align}
Set
\begin{align}
&\vf^-_0 (\bx) = q_0 \sin\theta_{\rm i} \cos \theta_{\rm w} x_1
  -q_0 \sin\theta_{\rm i} \sin \theta_{\rm w}x_2 + q_0\cos\theta_{\rm i}x_3, \label{exp-phi0-}\\[1mm]
&\vf^+_0 (\bx) =u_0 x_1 +  q_0\cos\theta_{\rm i} x_3.\label{exp-phi0+}
\end{align}
Then
\begin{align}\label{2.25}
\vf_0 (\bx)
&=
\begin{cases}
\vf^-_0 (\bx) &\,\,\mbox{for $x_2 > \sigma x_1$},\\[1mm]
\vf^+_0 (\bx) &\,\, \mbox{for $x_2 < \sigma x_1$}
\end{cases}
\end{align}
is a piecewise constant transonic flow, which is used as our background solution
for the shock problem.

\subsection{$\,$ Free Boundary Problem and the Main Theorem}

Given the piecewise constant background solution with planar transonic shock,
we formulate the shock problem as the following
free boundary problem:

\begin{prob}[Free Boundary Problem]\label{prob}
$\,$ Let a background solution $(\vf^-_0, \vf^+_0)$ with a plane shock $\cs_0$ be given by \eqref{2.25}.
Suppose that the incoming flow $\vf^-$ defined in the domain{\rm :}
\begin{equation*}
\gO^-_{\rm e} = \left\{\bx \in \R^3\,:\,\frac{\sigma}{2}(x_1 -e_1(x_3) )<x_2 -e_2(x_3)< 2 \sigma
(x_1 -e_1(x_3)) \right\}
\end{equation*}
is a small perturbation of  $\vf^-_0$ and
satisfies equation \eqref{eqn-nondiv}.
Suppose that $\cw$ defined by \eqref{def-wedge} in domain \eqref{def-wedgedomain}
is a small perturbation of the half $(x_1,x_3)$--plane with positive $x_1$-axis.
Let
\begin{equation*}
 \cf_{e_2}:=\left\{(x_2,x_3) \in \R^2\,:\,x_2>e_2(x_3)\right\}.
\end{equation*}
Find a subsonic solution $\vf^+$ and a shock surface $($as a free boundary$)${\rm :}
\begin{align}\label{def-shock}
\cs:= \big\{\bx \in \R^3\,:\,(x_2,x_3) \in \cf_{e_2} ,  x_1= s(x_2,x_3)\big\} \qquad \mbox{with $e_1(x_3) = s(e_2(x_3),x_3)$},
\end{align}
which are also small perturbations of both the background solution $\vf_0^+$ and the plane shock $\cs_0$,
such that
\begin{enumerate}
\item The Rankine-Hugoniot conditions \eqref{con-RH1}--\eqref{con-RH2} as free boundary conditions hold on the shock surface $\cs$.
\item $\vf^+$ satisfies equation \eqref{eqn-nondiv} in the subsonic domain{\rm :}
\begin{align}\label{def-domain}
\gO_{\rm s}:= \big\{\bx \in \R^3\,:\,(x_1,x_3) \in \cd_{e_1} , w(x_1,x_3) < x_2,\, s(x_2,x_3)<x_1 \big\}.
\end{align}
\item The slip boundary condition holds on the wedge surface $\mathcal{W}${\rm :}
\begin{align}\label{con-slip}
\DD \vf^+ \cdot \bn |_{\cw} =0,
\end{align}
where $\bn = (-w_{x_1},1, -w_{x_3} )$ is a normal direction on $\cw$.
\end{enumerate}
\end{prob}

In order to measure the perturbations described in {Problem \ref{prob}} precisely,
we first
introduce the following weighted H\"older norms:
Let $\gO$ be an open domain in $\R^3$ with edge $\cE\subset\partial\gO$.
For any $\bx, \bx' $ in  $\gO$,  define
\begin{align}
&\gd_\bx :=  \min
(\textrm{dist}(\bx,\cE),1),&&\gd_{\bx,\bx'} := \min (\gd_\bx,\gd_{\bx'}),\label{def:2.27} \\
&\gD_\bx := \sqrt{x_1^2 +x_2^2}+1, &&\gD_{\bx, \bx'}:=
\min(\gD_\bx, \gD_{\bx'}).
\end{align}
Let $\ga \in (0,1)$,  $\tau,l\in \R$, and $k$ be a nonnegative integer.
Let $\bk = (k_1, k_2, k_3)$ be an integer-valued vector with $k_1, k_2, k_3 \ge 0$
and $|\bk|=k_1 +k_2 + k_3$,
and let $\DD^{\bk}= \partial_{x_1}^{k_1}\partial_{x_2}^{k_2}\partial_{x_2}^{k_2}$.
We define
\begin{align}
[ f ]_{k,0;(l);\gO}^{(\tau)}
&= \sup_{ \begin{subarray}{c}
	\bx\in \gO\\
	|\bk|=k
	\end{subarray}}\big((\gd_{\bx})^{\max (k+\tau,0)} \gD_\bx^{l+k}  |\DD^\bk f(\bx)|\big),\label{def-normk0}\\
{[ f ]}_{k,\ga;(l);\gO}^{(\tau)}&=
\sup_{
	\begin{subarray}{c}
	\bx, \bx'\in \gO, \bx \ne \bx'\\
	|\bk|=k
	\end{subarray}}
\begin{array}{l}
\Big((\gd_{\bx,\bx'})^{\max(k+\ga+
	\tau,0)}  \gD_{\bx,\bx'}^{l+k+\ga}\frac{|\DD^\bk f(\bx)-\DD^\bk
	f(\bx')|}{|\bx-\bx'|^\ga}\Big),
\end{array}\label{def-normkga}
\\
\|f\|_{k,\ga;(l); \gO}^{(\tau)}&= \sum_{i=0}^k {[ f ]_{k,0;(l);\gO}^{(\tau)}} + {[ f ]}_{k,\ga;(l);\gO}^{(\tau)}. \label{def-norm}
\end{align}

Similarly, for a domain  $\gO$ in the $(x_1,x_3)$--plane with boundary $\cE$
close to the $x_3$--axis,  we modify the definition of $\gD_{\bx}$ to
\begin{align*}
\gD_{\bx}:= |x_1|+1 \qquad \text{ for $\bx= (x_1,x_3) \in \gO$},
\end{align*}
and the definition of norms \eqref{def-normk0}--\eqref{def-norm} above applies to
the functions defined on $\gO$ in the $(x_1,x_3)$--plane.

Finally, for either the $\R^3$ case or the $(x_1,x_3)$--plane case, denote the function space:
\begin{equation}\label{def-Ckga}
C^{k,\alpha;(l)}_{(\tau)}(\Omega) = \{ f: \|f\|_{k,\alpha;(l); \Omega}^{(\tau)}< \infty \}.
\end{equation}

Then the main theorem of the paper is stated as

\begin{thm}[Main Theorem] \label{thm-main}
$\,$ There are $\ga\in (0,1), \gb\in (0,1),  C_0>0$, and $\varepsilon>0$ depending only on the background state $(\vf^-_0, \vf^+_0)$
such that,  if
\begin{align}\label{2.17}
\|\vf^- - \vf^-_0\|_{3,\ga;(-\gb);\gO^-_e}^{(-2-\ga)} + \|e_1\|_{2,\ga;\R}+ \|w\|_{3,\ga;(-\gb);\cd_{e_1}}^{(-2-\ga)} \le \ve,
\end{align}
then there exist an attached weak transonic shock $\cs$ and a subsonic solution $\vf^+$  for
{Problem {\rm \ref{prob}}} satisfying the following estimate{\rm :}
\begin{align}
\begin{split}
&\|\vf^+ - \vf^+_0\|_{2,\ga;(-\gb);\gO_{\rm s}}^{(-1-\ga)} +\|s - s_0\|_{2,\ga;(-\gb);\cf_{e_2}}^{(-1-\ga)} \\
&\le C_0\Big(\|\vf^- - \vf^-_0\|_{2,\ga;(-\gb);\gO^-_e}^{(-1-\ga)} + \|e_1\|_{1,\ga;\R}+ \|w\|_{2,\ga;(-\gb);\cd_{e_1}}^{(-1-\ga)}\Big),
 \end{split}\label{est-thm}
\end{align}
where $w(x_1,x_3)$ and $e_1(x_3)$ are defined in \eqref{def-wedge}--\eqref{def-wedgedomain}, and $s_0(x_2)=\sigma^{-1}x_2$.
The solution satisfying estimate \eqref{est-thm} is unique.
In addition, if $|\vf^--\vf^-_0|+|\DD\vf^--\DD\vf^-_0|+|e_1|+|w|\rightarrow0$
as $x_3\rightarrow\pm \infty$
{\rm (}or as $x_3\rightarrow-\infty${\rm )}
pointwise,
then $|\vf^+-\vf^+_0|+|\DD\vf^+-\DD\vf^+_0|+|s-s_0|\rightarrow0$ as $x_3\rightarrow\pm \infty$
{\rm (}or as $x_3\rightarrow-\infty${\rm )}
pointwise correspondingly.
\end{thm}

\begin{rem}
$\,$ More precisely, in Theorem {\rm \ref{thm-main}}, the pointwise convergence means that,
if
\begin{align*}
&\lim_{x_3\rightarrow\pm\infty}e_1(x_3)=0,
\qquad
\lim_{x_3\rightarrow\pm \infty, x_1>0}w(x_1, x_3)=0,\\
&\lim_{\substack{x_1>0, \frac{\sigma}{2}x_1<x_2<2\sigma x_1\\x_3\rightarrow\pm \infty}}\big(|(\varphi^{-}-\varphi^{-}_0)(\mathbf{x})|
 +| (D\varphi^{-}-D\varphi^{-}_0)(\mathbf{x}) |\big) =0,
\end{align*}
then, correspondingly,
\begin{align*}
&\lim_{x_3\rightarrow\pm \infty, x_2>0}|(s-s_0)(x_2, x_3)|=0, \\
&\lim_{\substack{x_2>0,\pm(x_1-s_0(x_2))>0\\x_3\rightarrow\pm \infty}}\big(|(\varphi^{+}-\varphi^{+}_0)(\mathbf{x})|
+ |(D\varphi^{+}-D\varphi^{+}_0)(\mathbf{x})|\big) =0.
\end{align*}
\end{rem}

\smallskip
\begin{rem}
$\,$ In order to use the Banach fixed-point
theorem,
we assume the higher regularity with the norms in \eqref{2.17} than those in \eqref{est-thm},
due to the coordinate transformation introduced in {\rm \S \ref{subsec:coordinate}}
to flatten the wedge surface.
The main reason is that, after the coordinate transformation,
the coefficients in the equations and the boundary conditions depend on the derivatives
of the wedge surface.
We face the same situation when applying the partial hodograph transformation
if the wedge surface is not flat,
since the coefficients of the lower order terms depend also
on the derivatives of the wedge surface.
Therefore, the higher regularity in \eqref{2.17} for $w(x_1,x_3)$ and $e_1(x_3)$ is essential
in order to employ the contraction mapping theorem.
\end{rem}

\begin{rem}
$\,$ In {\rm \cite{CF}}, the stability of the piecewise constant weak transonic flow
is obtained in the sense that the perturbations are away from the wedge edge.
In this paper, we develop a different approach to remove the restriction
such that the structure is stable with respect to arbitrary small perturbations
of the wedge edge and the incoming flow up to the wedge edge.
\end{rem}

\begin{rem}
$\,$ From Theorem {\rm 2.1}, we can obtain the asymptotic behavior of the weak shock $\cs$ and
the subsonic solution $\varphi^+$ of {Problem {\rm \ref{prob}}}.
In fact, based on the definition of the weighted norm in \eqref{def:2.27}--\eqref{def-norm},
it follows from \eqref{est-thm} that $\DD\varphi^+$ converges to $\DD\varphi_0^+$ in $\Omega_{\rm s}$
with the decay rate $\Delta_{\bx}^{\beta-1}$ as $\Delta_{\bx}\to \infty$,
and the slope of the shock surface $\cs$
converges to the slope of $\cs_0$ with the same decay rate as $\Delta_\bx\to \infty$.
\end{rem}

\begin{rem}
$\,$ We establish the existence of solutions $\varphi^+$ that are uniformly bounded in the $x_3$--direction
and sublinearly grow in the $(x_1,x_2)$--directions.
This observation allows us to construct a barrier function first to show the uniqueness of solutions of
the linear problem $($see the proof of Theorem {\rm \ref{lem-linear}} below$)$, and then to show the uniqueness
of solutions of the nonlinear problem by the contraction mapping theorem.
\end{rem}

\section{$\,$ Mathematical Reformulation} \label{sec-iteration}
In this section, we reformulate the free boundary problem, {Problem \ref{prob}},
by introducing the coordinate transformation to fix the domain, and then introduce an iteration scheme in the new coordinates.
In other words, to solve {Problem \ref{prob}}, we follow the procedure as described below.

\subsection{$\,$ Background Solutions: Piecewise Constant Transonic Flow}\label{subsec:background}

For $(U_0^-, U_0^+)$ defined in \eqref{def-u0-}--\eqref{def-u0+},
condition \eqref{con-RH3} gives rise to
\begin{align}
\rho(q_0^2 )q_0^2\sin^2\theta_{\rm i}
- q_0 u_0\big(\rho(q_0^2 )+ \rho(u_0^2 +q_0^2\sin^2\theta_{\rm i})\big)\sin \theta_{\rm i}\cos \theta_{\rm w}
+ \rho(u_0^2 +q_0^2\sin^2\theta_{\rm i})u_0^2 =0. \label{eqn-utheta}
\end{align}
This implies as in \cite[\S 2]{CF} that, for a fixed incoming flow $U_0^-$,
the possible downstream states $U_0^+$ connecting by a shock form a balloon.
Next, for a fixed wedge edge (\emph{i.e.}, the $x_3$--axis),
the possible downstream velocity states together form a curve that is the intersection between the balloon and the plane orthogonal
to the wedge edge
(see Fig. \ref{fig:3dpolar}: \emph{all the red lines and curve lie in a plane orthogonal to the edge
with the plane also in red}).
The properties of the curve are similar to the corresponding shock polar in the two-dimensional case.
Moreover, on the curve,  $P_{\rm s}$ is the sonic point and $P_{\rm t}$ is the detached point.
The detached point $P_{\rm t}$ divides the two-dimensional shock polar curve into two subarches.
Let $\theta_{\rm w}^*$ be the dihedral wedge-angle such that the corresponding wedge plane intersects
the curve at $P_{\rm t}$.
Clearly, when $\theta_{\rm w}>\theta_{\rm w}^*$, there is no intersection point between the wedge plane and the curve,
which means that there is no attached plane shock for this case.

When $\theta_{\rm w}<\theta_{\rm w}^*$, the wedge plane intersects with the curve at two points: One of them
corresponding to the higher speed is the weak shock solution, and the other is the strong shock solution.
Let $U_0^+$ be the intersection point lying in the subarch corresponding to the weak shock solution.
Next, let $\theta_{\rm s}^*>0$ be the sonic angle such that $U_0^+=P_{\rm s}$ when $\theta=\theta_{\rm s}^*$.
By \cite[in \S 2]{CF}, such $\theta_{\rm s}^*$ exists with $\theta_{\rm s}^*\in(0,\theta_{\rm w}^*)$.
We know that the weak shock solution is transonic if $\theta_{\rm w}\in(\theta_{\rm s}^*,\theta_{\rm w}^*)$
and is supersonic if $\theta_{\rm w}\in(0,\theta_{\rm s}^*)$.
Therefore, the shock solution $(U_0^-,U^+_0)$ is a weak transonic solution if $\theta_{\rm w}\in(\theta_{\rm s}^*,\theta_{\rm w}^*)$.

\begin{figure}
	\centering
	\includegraphics[width=0.7\textwidth]{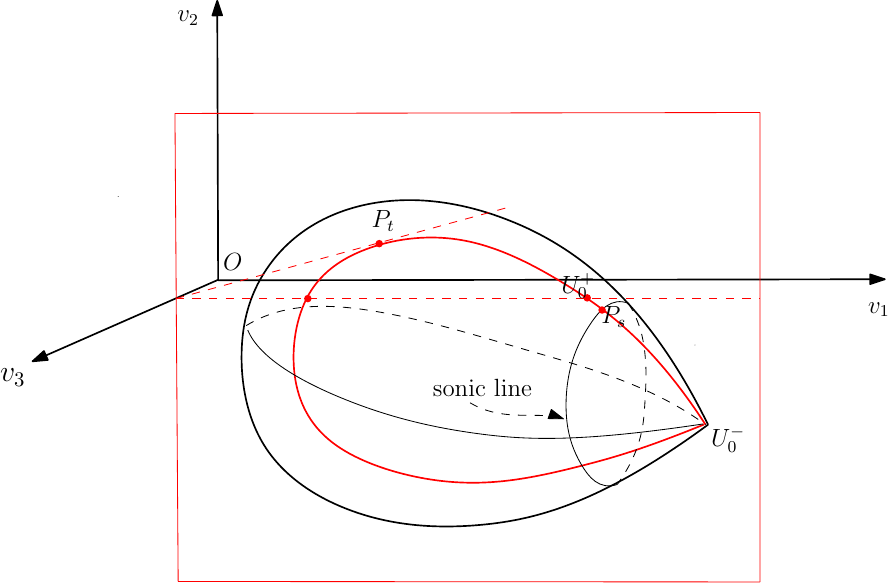}
	\protect\caption{Shock polar for the three-dimensional case.\label{fig:3dpolar}}
\end{figure}

\subsection{$\,$ Coordinate Transformation}\label{subsec:coordinate}
Given a function $s(x_2,x_3)$ close to $s_0(x_2)=\sigma^{-1}x_2$,
we first define the shock surface $\cs$ as a free boundary  $x_1=s(x_2,x_3)$,
and then define $\Omega_{\rm s}$ by \eqref{def-domain}.
Moreover, \eqref{con-RH1} holds on the shock surface, where $\vf^+$ is replaced by $\vf$:
\begin{equation*}
\varphi(s(x_2,x_3),x_2,x_3)=\varphi^-(s(x_2,x_3),x_2,x_3).
\end{equation*}

Notice that,
for each different function $s(x_2,x_3)$,
the shock surface is different, so that the boundary condition \eqref{con-RH2} is on the different surface.
Therefore, we need to further introduce the coordinate transformation to fix the shock surface.

Let
\begin{align*}
\cw_0:=\{\mathbf{y}:\, y_1 >0,\,  y_2 =0,\,y_3\in\mathbb{R}\},\qquad
\cE_0:=\{\mathbf{y}:\, y_1= y_2 =0,\,y_3\in\mathbb{R}\}.
\end{align*}
For our convenience, we first define a coordinate transformation to flatten the boundary of $\gO_{\rm s}$ and map $\gO_{\rm s}$
into the domain{\rm :}
\begin{equation}\label{eq:3.5}
\Pi:=\{\mathbf{y}:\, 0<y_2 < \sigma y_1 ,\,y_3\in\mathbb{R}\}
\end{equation}
bounded by the plane shock $\cs_0$ and the straight wedge $\cw_0$ with edge $\cE_0$.

Given the wedge function $w$ and the shock $\cs$, define the new coordinates in the following way:  First, let
\begin{equation}\label{def-transy2y3}
y_2=x_2-w(x_1,x_3),
\qquad y_3=x_3,
\end{equation}
where
$x_2=w(x_1,x_3)$ is the wedge surface,
and $(e_1(x_3),w(e_1(x_3),x_3),x_3)$ is the wedge edge.
Then the shock surface becomes
\begin{align*}
x_1 = s(y_2 + w(x_1, y_3), y_3).
\end{align*}	
Solving for $x_1$ gives
\begin{align*}
x_1 = \hs(y_2 , y_3).
\end{align*}
	
Let
\begin{align*}
 \gd \hs (y_2,y_3)& :=\hs(y_2,y_3) - s_0 (y_2) = \hs(y_2,y_3) - \sigma^{-1}y_2.
\end{align*}
Then we extend $\delta\hat{s}(y_2,y_3)$ to be a function of $\mathbf{y}=(y_1,y_2,y_3)$ by defining
\begin{align}\label{exp-extend}
\gd \bar{s}(\by) &:=\eta(\sigma y_1 -y_2) \int_{\R} \gd \hs(y_2 , y_3+ t(\sigma y_1 -y_2))	\xi(t)\, \mathrm{d}t ,
\end{align}
where $\xi(t)$ is a smooth mollifier satisfying
\begin{align*}
 \xi(t)\geq 0, \quad \mathrm{supp}\,  \xi(t) \subset [-1,1] ,  \quad \int_{\mathbb{R}}\xi(t)\,{\rm d}t=1,
\end{align*}
and  $\eta \in C_c^{\infty}(\mathbb{R})$ is a cutoff function with the following properties:
\begin{align*}
\begin{cases}
0 \le \eta(t)\le 1\quad &\mbox{for $t\in \R$}, \\
 \eta(t)=1 &\mbox{for $t\in [-1,1]$},\\
  \eta(t)=0 &\mbox{for $t\in (-\infty, -2]\cup [2,\infty)$}.
\end{cases}
 \end{align*}
Let
\begin{align*}
\cq = \{\by: y_1>0, y_2 >0\},\qquad  \cf_0 =  \{(y_2, y_3): y_2 >0\}.
\end{align*}
By definition \eqref{exp-extend}, we have
\begin{align}
&\gd \ls|_{\cs_0} =\gd \hs, \quad \mathrm{supp} (\gd \ls)\subset \{\by: -1 \le \sigma y_1 -y_2 \le 1\},\label{est-extend-a}\\
&\| \gd \ls \|_{2,\ga;(-\gb);\cq}^{(-1-\ga)} \le C \| \gd \hs \|_{2,\ga;(-\gb);\cf_0}^{(-1-\ga)},
\label{est-extend}
\end{align}
where the norm in \eqref{est-extend} is defined in \eqref{def-norm} via replacing $\cE$
by edge $\cE_0=\{(y_1,y_2,y_3):\,y_1=y_2=0\}$ in $\mathbb{R}^3$,
or by edge $\cE_0=\{(y_2,y_3):\,y_2=0\}$ in $\mathbb{R}^2$.
See Lemma 2.3 in \cite{GH} for the details of the proof of \eqref{est-extend}.

Define $y_1$ implicitly by
\begin{align}\label{def-transy1}
x_1& = y_1 + \gd \bar{s}(\by)  =  y_1 + \gd \bar{s}(y_1,x_2-w(x_1,x_3), x_3 ) .
\end{align}
Then we have

\begin{lem}\label{lem:3.1}
$\,$ If $\DD w$ and $\DD(\delta\hat{s})$ are sufficiently small in the $C^0$--norm,
then the coordinate transformation \eqref{def-transy2y3} and \eqref{def-transy1} is invertible.
Moreover, $\Omega_{\rm s}$ in the $\mathbf{x}$--coordinates defined by \eqref{def-domain}
is mapped into domain $\Pi$ in the new $\mathbf{y}$--coordinates defined by \eqref{eq:3.5}.
\end{lem}

\begin{proof}
$\,$ First, when $w=0$ and $\delta\hat{s}=0$, we know that the Jacobian matrix
of the coordinate transformation $J_{\bx}\by$ is the identity matrix $I_3$.
Thus, the coordinate transformation is invertible
when $\DD w$ and $\DD(\delta\hat{s})$ are small.
	
Next, it is direct from definition \eqref{def-transy2y3} that $y_2\geq0$.
Since it is easy to see that $y_2\leq \sigma y_1$,
it suffices to show that $y_2=\sigma y_1$ on the shock.
By the definition, we know that, on the shock,
$$
\delta\hat{s}+ \sigma^{-1}y_2=\hat{s}=x_1=y_1+\delta\bar{s}.
$$
The identity holds if and only if $y_2=\sigma y_1$, thanks to the implicit function theorem and
the fact that $\DD(\delta\hat{s})$ is sufficiently small in the $C^0$--norm.
This completes the proof.
\end{proof}

\subsection{$\,$ Iteration Scheme in the New Coordinates}
Based on Lemma \ref{lem:3.1}, it suffices to consider the problem in the fixed domain $\Pi$.
In the new $\mathbf{y}$-coordinates, we define
\begin{align*}
\Pi^-:=\big\{\mathbf{y}:\, y_1\geq0,\,\frac{2\sigma}{3}y_1\leq y_2\leq \frac{3\sigma}{2}y_1,\,y_3\in\mathbb{R}\big\}.
\end{align*}

Denote
\begin{align}
&\varphi_{\hs}^-(\mathbf{y})=\varphi^-(\mathbf{x}(\mathbf{y})),&& \gd \varphi_{\hs}^-(\mathbf{y})
   = \varphi_{\hs}^-(\mathbf{y})-\varphi^-_0(\mathbf{x}(\mathbf{y})) &&\mbox{ for $\by \in\Pi^-$},\label{def-dphiy-a}\\
&\varphi(\mathbf{y})=\varphi(\mathbf{x}(\mathbf{y})),&& \gd \varphi(\mathbf{y})
= \varphi(\mathbf{y})-\varphi^+_0(\mathbf{x}(\mathbf{y})) &&\mbox{ for $\by \in\Pi$}. \label{def-dphiy}
\end{align}
Define the iteration set $\mathcal{K}=\mathcal{K}_1 \times \mathcal{K}_2 $ by
\begin{align}
\mathcal{K}_1&:=\Big\{\gd \hs \in C^{2,\ga;(-\gb)}_{(-1-\ga)}(\cf_0)\,:\,\gd \hs(0,y_3)= e_1(y_3),\
               \|\gd \hs \|_{2,\ga;(-\gb);\cf_0}^{(-1-\ga)}\le C_0\ve\Big\},\label{def-K1}\\[1mm]
\mathcal{K}_2
&:=\Big\{ \gd \varphi \in  C^{2,\ga;(-\gb)}_{(-1-\ga)}(\Pi)\,:\,\|\gd \varphi\|_{2,\ga;(-\gb);\Pi}^{(-1-\ga)}\le C_0\ve\Big\}.\label{def-K2}
\end{align}
The estimate in \eqref{def-K1} for $\gd \hs$ guarantees that the shock surface in the $\by$--coordinates
stays in $\Pi^-$ if $\ve$ is sufficiently small.

In the $\mathbf{y}$--coordinates, \eqref{eqn-nondiv} becomes
\begin{equation*}
\sum_{i,j=1}^3\tilde{a}_{ij}^{\hs}(\by, \DD_{\by}(\gd \vf))\partial_{y_iy_j}(\gd \varphi)
+\sum_{i=1}^3\tilde{b}_i^{\hs}(\by, \DD_{\by}(\gd \vf))\partial_{y_i}(\gd \varphi)=0,
\end{equation*}
where
\begin{align}
&\tilde{a}_{ij}^{\hs}(\by, \DD_{\by}(\gd \vf))
=\sum_{k,m=1}^3a_{km}(\DD_{\bx}(\gd \vf) + U^+_0)\frac{\partial y_i}{\partial{x_k}}\frac{\partial y_j}{\partial{x_m}},\label{def-atilde}
\\
&\tilde{b}_i^{\hs}(\by, \DD_{\by}(\gd \vf))
  =\sum_{k,m=1}^3a_{km}(\DD_{\bx}(\gd \vf) + U^+_0)\frac{\partial^2 y_i}{\partial{x_k}\partial{x_m}},\label{def-btilde}\\
&\partial_{x_i}(\gd \vf) =\sum_{j=1}^3  \partial_{y_j}(\gd \vf)\frac{\partial y_j}{\partial{x_i}}, \qquad\,\, U^+_0=  \DD_{\bx}(  \vf_0^+(\bx)).
\end{align}
Hereafter, we write operator $\DD$ for $\DD_{\by}$ in the $\by$--coordinates when no confusion arises.

To solve the fixed boundary value problem above, we linearize the equation and the boundary conditions,
and then make careful uniform estimates required in order
to apply the Banach fixed point theorem.
More precisely, for any given function $(\gd \hs, \gd \vf)\in\mathcal{K}$,
we solve the following linear equation:
\begin{align}\label{eqn-vflinear}
\sum_{i,j=1}^{3}a_{ij}(U^+_0)(\gd \tvf)_{y_{i}y_{j}}= f^{\hs}(\DD (\gd \vf), \DD^2 (\gd \vf)),
\end{align}
where
\begin{align}\label{def-f}
f^{\hs}(\by, \DD (\gd \vf), \DD^2 (\gd \vf))= \sum_{i,j=1}^{3}\big(a_{ij}(U^+_0)-\tilde{a}_{ij}^{\hs}\big) (\gd \vf)_{y_{i}y_{j}}-\sum_{i=1}^{3}\tilde{b}_i^{\hs}(\gd \varphi)_{y_i}.
\end{align}

Condition \eqref{con-slip} on boundary  $\cw$ can be rewritten as
\begin{align*}
\DD_{\bx} (\gd \vf)\cdot \bn| _{\cw} = -U_0^+\cdot \bn = u_0 w_{x_1} +q_0\cos \theta_{\rm i} w_{x_3},
\end{align*}
or
\begin{align*}
\DD_{\by} (\gd \vf)J_{\bx}\by\cdot \bn| _{\cw_0} = -U_0^+\cdot \bn  \qquad
\,\,\mbox{in the $\by$--coordinates}.
\end{align*}
Set the condition on $\cw_0$ for equation \eqref{eqn-vflinear} as follows:
\begin{align}\label{con-wedgelinear}
(\gd \tvf)_{y_2}|_{\cw_0} = g_{\rm w}^{\hs}(\by ,\DD(\gd \varphi) )|_{\cw_0},
\end{align}
where
\begin{align}\label{def-gw}
 g_{\rm w}^{\hs}(\by ,\DD(\gd \varphi) )= (\gd \vf)_{y_2}-( \DD_{\by} (\gd \vf)J_{\bx}\by +U_0^+)\cdot \bn
\end{align}
with $\bn=(w_{x_1},-1,w_{x_3})$ and
\begin{align*}
w_{x_i}|_{\cw_0}= \DD_{\by}w(y_1 + \gd \ls(y_1,0,y_3), y_3)J_{\bx}\by \qquad\,\, \mbox{for $i=1,3$}.
\end{align*}

Finally, we rewrite condition \eqref{con-RH3} on the shock surface $\cs$ in the $\bx$--coordinates into
\begin{align}\label{con-gtildes}
\DD_{\bx}(\gd \vf)\cdot \bmu|_{\cs} = \tilde{g}_{\rm s}(\bx, \DD_{\bx}(\gd \varphi))|_{\cs}
\end{align}
with
\begin{align}
&\tilde{g}_{\rm s}(\bx, \DD_{\bx}(\gd \varphi))=  \DD_{\bx}(\gd \vf)\cdot \bmu - H(\DD_{\bx}\vf^-, \DD_{\bx}(\gd \varphi)+ U_0^+),\label{def-gstilde}\\
&\bmu =(\mu_1,\mu_2, \mu_3)=H_{\bv}(\DD_{\bx}\vf^-_0(\bx), \DD_{\bx}\vf_0^+(\bx)).\label{def-mu}
\end{align}
We write condition \eqref{con-gtildes} in the $\by$--coordinates as
\begin{align}\label{con-gs}
\DD_{\by}(\gd \vf)J_{\bx}\by\cdot \bmu|_{\cs_0} = \tilde{g}_{\rm s}(\bx (\by), \DD_{\by}(\gd \varphi)J_{\bx}\by)|_{\cs_0}.
\end{align}
Thus, we impose the following oblique derivative boundary condition:
\begin{align}\label{con-shocklinear}
\DD_{\by}(\gd \tvf)\cdot \bmu|_{\cs_0} = g_{\rm s}^{\hs}(\by, \DD_{\by}(\gd \varphi))|_{\cs_0},
\end{align}
where
\begin{align}\label{def-gs}
 g_{\rm s}^{\hs}(\by, \DD_{\by}(\gd \varphi))&=\DD_{\by}(\gd \vf)(I-J_{\bx}\by)\cdot \bmu +\tilde{g}_{\rm s}(\bx (\by), \DD_{\by}(\gd \varphi)J_{\bx}\by).
\end{align}

In order to keep the shock surface attached to edge $\cE$, one of the main ingredients in our new approach
is to impose an extra condition on $\cE$:
\begin{align}\label{con-edge}
(\tvf -  \vf^-)|_{\cE} = 0,
\end{align}
which can be written as
\begin{align}\label{con-edge1}
\gd \tvf|_{\cE_0} =g_{\rm e}^{\hs}
\end{align}
in the $\by$--coordinates, where
\begin{align}\label{def-ge}
g_{\rm e}^{\hs}(y_3)=\big(\vf^-(\bx(\by)) -\vf^+_0(\bx(\by))\big)(0, 0,y_3).
\end{align}

Denote
\[
a_{ij}^0 = a_{ij}(U^+_0).
\]
We first solve the following linear problem:

\begin{prob}[Linearized Fixed Boundary Problem]\label{prob-linear}
$\,$ Given functions $f_1$ defined in $\Pi$, $ g_1$ in $\cd_0$, $g_2$ in $\cf_0$, and $g_3$ in $\R$, solve the equation{\rm :}
\begin{align}\label{eqn-shocklinear}
\sum_{i,j=1}^{3}a_{ij}^0 v_{y_{i}y_{j}}= f_1 \qquad\,\mbox{in $\Pi$}
\end{align}
with the boundary conditions{\rm :}
\begin{align}
&v_{y_2}|_{\cw_0} = g_1, \label{con-wg1}\\
&\DD v\cdot \bmu|_{\cs_0} = g_2,\label{con-sg2} \\
& v|_{\cE_0} =g_3.\label{con-eg3}
\end{align}	
\end{prob}

We solve {Problem \ref{prob-linear}}
by proving the following theorem in \S 4:

\begin{thm}\label{lem-linear}
$\,$ Assume equation \eqref{eqn-shocklinear} is uniformly elliptic, $\mu_1 >0$,  $\mu_2 >0$,
and  $\bmu\cdot\bn_{\rm sh}>0$,
where $\bn_{\rm sh}$ is the outer unit normal of the shock surface $\cs_0$.
Suppose that the angle between $\cw_0$ and $\cs_0$ is $\omega \in (0,\frac{\pi}{2})$.
Then there are $\ga,\gb \in (0,1)$ depending on $(a_{ij}^0, \bmu, \omega)$ such that,
if
$f_1 \in C^{0,\ga;(2-\gb)}_{(1-\ga)}(\Pi), g_1\in C^{1,\ga;(1-\gb)}_{(-\ga)}(\cd_0), g_2 \in C^{1,\ga;(1-\gb)}_{(-\ga)}(\cf_0)$,
and $g_3 \in C^{1,\ga}(\R)$,
there exists a unique  solution $v \in C^{2,\ga;(-\gb)}_{(-1-\ga)}(\Pi) $ of {Problem {\rm \ref{prob-linear}}}
with the following estimate{\rm :}
\begin{align}\label{est-v}
\|v\|_{2,\ga;(-\gb);\Pi}^{(-1-\ga)}
\le C \Big(\|f_1\|_{0,\ga;(2-\gb);\Pi}^{(1-\ga)} + \|g_1\|_{1,\ga;(1-\gb);\cd_0}^{(-\ga)}
+ \|g_2\|_{1,\ga;(1-\gb);\cf_0}^{(-\ga)} +  \|g_3\|_{1,\ga;\R}\Big),
\end{align}	
where $C>0$ is a constant depending on  $(a_{ij}^0, \bmu, \omega,\ga,\gb)$.
\end{thm}

\begin{rem}
$\,$ Propositions {\rm \ref{prop:5.1a}}--{\rm \ref{prop:5.2a}} later will guarantee
the assumptions in Theorem {\rm \ref{lem-linear}}
for the weak transonic shock.
In fact, the non-obliqueness assumption, {\it i.e.}, $\mu_1>0$ and $\mu_2>0$ in Theorem {\rm \ref{lem-linear}},
allows us to assign the boundary condition \eqref{con-eg3} on the wedge edge,
which means that the shock is an attached shock.
It is the key difference from the strong transonic shock, where the non-obliqueness assumption fails.
That is the mathematical reason why we expect the weak transonic shock is stable {\rm \cite{CF}},
but the strong transonic shock is unstable {\rm \cite{LXY}}, for the attached plane shock over a three-dimensional flat wedge
with respect to the three-dimensional perturbations.
\end{rem}

\begin{rem}
$\,$ {Problem {\rm \ref{prob-linear}}} is for the linear stability of {Problem {\rm \ref{prob}}}.
Thus, if Theorem {\rm \ref{lem-linear}} is proved, then we conclude that {Problem {\rm \ref{prob}}} is linearly stable.
\end{rem}

It follows from Theorem \ref{lem-linear} that,
given $(\gd \hs, \gd \vf)\in\mathcal{K}$,
we solve {Problem \ref{prob-linear}} with
$v=\gd \tvf, f_1=
f^{\hs}$, $g_1=g_{\rm w}^{\hs}
$, $g_2=g_{\rm s}^{\hs}$, and $g_3=g_{\rm e}^{\hs}$.

\smallskip
Now, we use condition \eqref{con-RH1} to update the shock function in the following way:
Solve the equation for $x_1$:
\begin{align}\label{eqn-updateshock}
(\vf^- -\vf^+_0)(x_1, y_2+w(x_1,y_3),y_3) = \gd \tvf(\sigma^{-1}y_2,y_2,y_3)
\end{align}	
by the implicit function theorem to obtain $x_1 = \ts(y_2,y_3)$.
Set $\gd \ts:= \ts - s_0$.
Define a map $\ft$ on $\mathcal{K} $ by
 \begin{equation}\label{def:T}
\ft(\gd \hs, \gd \vf)=(\gd \ts, \gd \tvf).
\end{equation}

We will show that the assumptions in Theorem \ref{lem-linear} hold and establish
the required uniform estimates for applying the Banach fixed-point
theorem
to ensure the existence of a fixed point
of map $\ft$ in \S \ref{sec-6} to conclude the proof of Theorem \ref{thm-main}.

\section{$\,$ Linear Stability: Proof of Theorem \ref{lem-linear}}

Without loss of generality, in {Problem \ref{prob-linear}},
we may assume
\begin{align}\label{con-g=0}
& g_1(0,y_3)= g_2(0,y_3) = g_3(y_3) \equiv 0,
\end{align}
where $g_1$, $g_2$, and $g_3$ are in the boundary conditions \eqref{con-wg1}--\eqref{con-eg3}.

Indeed, we can extend $g_3$ from domain $\cE_0 $ to $\Pi$ (\textit{cf.} Lemma 2.5 in \cite{GH}).
Denote the extended function by $\tilde{g}_3$ so that
\begin{align*}
&\mathrm{supp}\,\tilde{g}_3 \subset \big\{\by\in \Pi: y_1^2 +y_2^2 \le 1\big\},\quad (\tilde{g}_3 )_{y_2}|_{\cE_0}  = g_1|_{\cE_0} ,
\quad \DD \tilde{g}_3 \cdot \bmu|_{\cE_0}  = g_2|_{\cE_0}, \\[1mm]
&\| \tilde{g}_3 \|_{2,\ga;(-\gb);\Pi}^{(-1-\ga)}
\le C \Big(\|g_1\|_{1,\ga;(1-\gb);\cd_0}^{(-\ga)}  + \|g_2\|_{1,\ga;(1-\gb);\cf_0}^{(-\ga)}+  \|g_3\|_{1,\ga;\R}\Big).
\end{align*}
Then assumption \eqref{con-g=0} is satisfied if {Problem \ref{prob-linear}} is solved for $\tilde{v} = v- \tilde{g}_3$.

To solve {Problem \ref{prob-linear}}, we truncate domain $\Pi$ by  a ball $B_R(\bO)$,
centered at $\bO$ with radius $R$, so that we can work on a finite domain.
Furthermore, since $\mu_1 >0$ and $\mu_2 >0$,
conditions  \eqref{con-wg1}--\eqref{con-sg2} are not oblique at the wedge edge $\cE_0$ (\textit{cf.} \cite{Lieberman2}).
In order to resolve this difficulty, we also truncate the wedge edge.
It is convenient to use the cylindrical coordinates $(r, \theta, y_3)$ for the truncation and the estimates later.
More precisely, the truncation is given as follows:
\begin{align*}
\begin{split}
&\Pi^{R} = \big\{\by \in \Pi\cap B_R(\bO)\,:\, r>R^{-1} \big\},\quad \cw^{R} = \big\{\by \in \cw_0\cap B_R(\bO)\,:\,r> R^{-1}\big\},\\
& \cs^R =  \big\{\by \in \cs_0\cap B_R(\bO)\,:\, r >R^{-1}\big\},
\quad \ct^R = \partial \Pi^R \backslash {\big(\cw^R \cup \cs^R\big)},
\end{split}
\end{align*}
where  $R >4$.

\smallskip
Now we first solve the following problem in the truncated domain $\Pi^{R}$:

\begin{prob}[Problem in Truncated Domains]\label{prob-trunc}
$\,$ Given $f_1$, $g_1$, and $g_2$ as in {Problem {\rm \ref{prob-linear}}} with assumption \eqref{con-g=0},
solve the boundary value problem{\rm :}
\begin{align}\label{eqn-vlinear}
\sum_{i,j=1}^{3}a^0_{ij} v_{x_ix_j} = f_1  \qquad \text{in $\Pi^R$}
\end{align}
with boundary conditions{\rm :}
\begin{align}
& v_{x_2}|_{\cw^R} = g_1, \label{con-bdg1}\\
&\DD v\cdot \bmu|_{\cs^R} = g_2, \label{con-bdg2}\\
& v|_{\ct^R} = 0.\label{con-ct}
\end{align}
\end{prob}

For {Problem \ref{prob-trunc}}, we have the following lemma.

\begin{lem} \label{lem1}
Under the same assumptions for $a_{ij}^0$, $\bmu$, and $\omega$ as in Theorem {\rm \ref{lem-linear}},
there are $\ga,\gb \in (0,1)$  depending on $(a_{ij}^0, \bmu, \omega)$ such that,
for the same functions $(f_1, g_1, g_2)$ as in Theorem {\rm \ref{lem-linear}} with assumption \eqref{con-g=0},
there exists a unique solution  $v \in C^{2,\ga}(\Pi^R) \cap C^0(\overline{\Pi^R})$
for {Problem {\rm \ref{prob-trunc}}} satisfying the following estimate{\rm :}
\begin{align}\label{est-vtrunc}
\|v\|_{2,\ga;(-\gb);\Pi^{\frac{R}{2}}}^{(-1-\ga)}
\le C \Big(\|f_1\|_{0,\ga;(2-\gb);\Pi}^{(1-\ga)} + \|g_1\|_{1,\ga;(1-\gb);\cd_0}^{(-\ga)}  + \|g_2\|_{1,\ga;(1-\gb);\cf_0}^{(-\ga)}\Big),
\end{align}	
where $C>0$ is a constant depending on  $(a_{ij}^0, \bmu, \omega, \ga, \gb)$,
but independent of $R$, and the weights for the superscripts in \eqref{est-v} are to the wedge edge $\cE_0$.
\end{lem}

\begin{proof}
$\,$ We divide the proof into three steps.

\medskip
1. $\,$ Let
\begin{align*}
C^*(\Pi^R):= C^0(\overline{\Pi^{R}}) \cap C^2(\Pi^{R} \cup \cw^{R}\cup \cs^{R}).
\end{align*}
By Theorem 1 in \cite{Lieberman1}, there is a unique solution
$v_R\in C^*(\Pi^{R})$ for {Problem \ref{prob-trunc}}.
Then, to prove Lemma \ref{lem1}, it suffices to obtain the uniform estimate
\eqref{est-vtrunc}.

\smallskip
2. $\,$ Let
\begin{align}\label{def-M}
M := \|f_1\|_{0,\ga;(2-\gb);\Pi}^{(1-\ga)} + \|g_1\|_{1,\ga;(1-\gb);\cd_0}^{(-\ga)}  + \|g_2\|_{1,\ga;(1-\gb);\cf_0}^{(-\ga)}.
\end{align}
Then we need the following estimate, independent of $R$:
\begin{align}\label{est-vdecay}
|v_R(\by)| \le C M \min (r^{1+\ga}, r^\gb) \qquad\mbox{for $\by \in \Pi^{R}$}.
\end{align}

To achieve estimate \eqref{est-vdecay}, we scale $(y_1,y_2)$ into the following
coordinates $(\bar{y}_1, \bar{y}_2)$:
\begin{align}\label{def-coords}
(\bar{y}_1,\bar{y}_2) = (d_1y_1, d_2 y_2)
\qquad \mbox{with $(\bar{r}, \bar{\theta}) =((\bar{y}_1^2 + \bar{y}_2^2)^{\frac{1}{2}}, \arctan (\frac{\bar{y}_2 }{ \bar{y}_1}))$},
\end{align}
where $(d_1,d_2)=((a^0_{11})^{-\frac{1}{2}}, (a^0_{22})^{-\frac{1}{2}})$.

Set the comparison function in the following form:
\begin{align}\label{def-vstar}
v^* = \bar{r}^l \sin (t \bar{\theta} + \theta_0),
\end{align}
where $l$, $t$, and $\theta_0 >0 $ will be determined later such that $v^*$ is a barrier function.

A direct calculation shows
\begin{align*}
\sum_{i,j=1}^{3}a^0_{ij} v^*_{y_iy_j}={}& (\partial_{\bar{y}_1}^2 + \partial_{\bar{y}_2}^2 ) v^*= (l^2 -t^2) {\bar{r}}^{l-2}\sin (t \bar{\theta} + \theta_0),\\
v^*_{y_2}|_{\cw_0} ={} \bar{r}^{l-1}&d_2 t \cos \theta_0,\\[1mm]
\DD v^*\cdot \bmu|_{\cs_0}={} &\bar{r}^{l-1}\mu_1d_1\big(l \cos \bar{\omega} \sin (t \bar{\omega} + \theta_0 ) - t \sin \bar{\omega} \cos  (t \bar{\omega} + \theta_0 )\big)\\
& +\bar{r}^{l-1}\mu_2d_2\big(l \sin \bar{\omega} \sin (t \bar{\omega} + \theta_0 ) + t \cos \bar{\omega} \cos  (t \bar{\omega} + \theta_0 )\big)\\[1mm]
={} & \bar{r}^{l-1} \sqrt{\mu_1^2d_1^2 +\mu_2^2d_2^2}
\big((l-t)\cos (\bar{\omega}- \Phi)\sin (t \bar{\omega} + \theta_0 )
  + t \sin ( (t-1)\bar{\omega} + \theta_0 +\Phi )\big),
\end{align*}
where
\begin{align*}
\bar{\omega}= \arctan (\frac{d_2}{d_1} \tan \omega), \qquad
\Phi = \arctan (\frac{\mu_2 d_2}{\mu_1 d_1}).
\end{align*}

First, for a fixed $\gb \in (0,1)$, choose $l = \gb$, $t = \gb + \tau_0$, and $\theta_0 = \frac{\pi}{2} + \tau_0$ in \eqref{def-vstar}, and set
\begin{align*}
v_1 = CM v^* = CM \bar{r}^\gb \sin ((\gb + \tau_0) \bar{\theta} + \frac{\pi}{2} + \tau_0),
\end{align*}
where $\tau_0 >0$ is suitably small and $C$ sufficiently large, depending on  $(a_{ij}^0, \bmu, \omega, \gb)$.

Since $\bar{\omega}\in(0,\frac{\pi}{2})$, which follows from $\tan\omega>0$ and the fact that $d_1$ and $d_2$ are positive,
we can find $\tau_0>0$ and $\beta>0$ sufficiently small such that
$(\beta+\tau_0)\bar{\theta}+\frac{\pi}{2}+\tau_0\in(\frac{\pi}{2}+\tau_0,\pi-\tau_0)$ in $\Pi^{R}$.
Thus, following the computation above, we have
\begin{align}\label{est-v1domain}
&\sum_{i,j=1}^{3}a^0_{ij} (v_1)_{y_iy_j} \le{} -CM\bar{r}^{\gb-2}\tau_0^2\sin \tau_0 \le f_1
  = \sum_{i,j=1}^{3}a^0_{ij} (v_R)_{y_iy_j} \qquad \text{ in $\Pi^{R}$},\\
&(v_1)_{y_2}|_{\cw_0} \le {}  -CM\bar{r}^{\gb-1}d_2 \gb \sin \tau_0 \le g_1  = (v_R)_{y_2}|_{\cw_0},\\
&\DD v_1\cdot \bmu|_{\cs_0} \ge {}  CM \bar{r}^{\gb-1}\gb\sin (\min(\Phi, \frac{\pi}{2} - \Phi)) \ge g_2 = \DD v_R\cdot \bmu|_{\cs_0}, \label{est-v1s0}\\
&v_1|_{\ct^{R}} \ge{} 0   =v_R|_{\ct^{R}}.
\end{align}
Therefore, by the comparison principle, we conclude
\begin{align} \label{est-rgb}
|v_R(\by)| \le CM r^\gb \qquad\,\,
\mbox{for any $\by \in \Pi^{R} $}.
\end{align}

Second, we now show the estimate of solution $v_R(\by)$ related to the $C^{1, \alpha}$--regularity up to the wedge edge,
thanks to the assumptions that $\mu_1>0$ and $\mu_2>0$.
Choose $l=1+\ga$, $t = 1+\ga + \tau_1$ , and $\theta_0 = \frac{\pi}{2} + \tau_1$ in \eqref{def-vstar},
where $\alpha$ and $\tau_1$ are sufficiently small and positive constants depending on  $(a_{ij}^0, \bmu, \omega)$.
Define the following barrier function:
\begin{align*}
v_2= CM v^* = CM \bar{r}^{1+\ga} \sin ((1+\ga + \tau_1) \bar{\theta} + \frac{\pi}{2} + \tau_1).
\end{align*}

Following the computation argument from \eqref{def-vstar} to \eqref{est-v1domain}, we have
\begin{align}\label{est-v2domain}
&\sum_{i,j=1}^{3}a^0_{ij} (v_2)_{y_iy_j} \le{} -CM\bar{r}^{\ga-1}\tau_1^2\sin \tau_1 \le f_1
   = \sum_{i,j=1}^{3}a^0_{ij} (v_R)_{y_iy_j} \qquad \text{ in $\Pi^{R}$},\\
&(v_2)_{y_2}|_{\cw_0} \le {}  -CM\bar{r}^{\ga}d_2 \ga \sin \tau_0 \le g_1  = (v_R)_{y_2}|_{\cw_0},\\[1mm]
&\DD v_2\cdot \bmu|_{\cs_0} \ge {}  CM \bar{r}^{\ga}\ga\sin (\min(\Phi, \frac{\pi}{2} - \Phi)) \ge g_2 = \DD v_R\cdot \bmu|_{\cs_0}, \label{est-v2s0}\\
&v_2|_{\ct^{R}} \ge{} 0   =v_R|_{\ct^{R}}.
\end{align}
Thus, $v_2$ meets the conditions for the comparison principle so that
\begin{align} \label{est-r1ga}
|v_R(\by)| \le CM r^{1+\ga}\qquad\mbox{for any $\by \in \Pi^{R} $}.
\end{align}
Combining \eqref{est-rgb} with \eqref{est-r1ga} leads to estimate \eqref{est-vdecay}.

We remark that, for estimate \eqref{est-v2s0},
we use the assumptions that $\mu_1$ and $\mu_2$ are positive such that $\Phi\in(0,\frac{\pi}{2})$.
Since the assumptions are not correct for the strong transonic shock,
we cannot expect \eqref{est-r1ga} and then cannot expect the $C^{1,\ga}$--regularity of solutions
of the strong transonic shock near the edge generally.

3. $\,$ Based on estimate \eqref{est-vdecay},
the standard Schauder estimates, and the scaling argument lead to \eqref{est-vtrunc}.
We now sketch the proof for the self-containedness.

Let $E$ be a bounded domain, and let $u$ be a function defined on $E$.
Define the following norms:
\begin{align*}
\| u\|'_{k;E} =\sum_{j=0}^k d^j[ u]_{j,0;E}, \qquad
\| u\|'_{k,\ga;E}=\| u\|'_{k; E}+ d^{k+\ga}[ u]_{k,\ga;E},
\end{align*}
where $d= \mbox{diam}\, E$.

For any point $\by^0 \in \Pi^{R/2} $ with cylindrical coordinates $(r_0, \theta_0, y_3)$, it falls into one of the following three cases:
\begin{flalign*}
\textit{Case {\rm 1}.} & \quad \frac{\omega}{4} \le \theta_0  \le \frac{3 \omega}{4},&\\
\textit{Case {\rm 2}.} &\quad  \frac{3\omega}{4}< \theta_0  <\omega_0,&\\
\textit{Case {\rm 3}.}& \quad  0 < \theta_0 < \frac{\omega_0}{4}.&
\end{flalign*}

For \textit{Case {\rm 1}}, let
\begin{align*}
\hat{r} = \frac{r_0}{4} \sin \frac{\omega}{4}, \quad
B_1 = B_{\hat{r}}(\by^0),\quad
B_2 = B_{2\hat{r}} (\by^0).
\end{align*}
By the definition, it is easy to see that
$B_1\subset B_2\subset\Pi^{R}$.
By the Schauder interior estimate (\textit{cf.} Theorem 4.6 in \cite{gt}),
for solution $v\in C^{2,\ga}(\Pi^R)\cap C^0(\overline{\Pi^R})$ of {Problem \ref{prob-trunc}}, we have
\begin{equation}\label{est-vinterior}
\|v \|'_{2,\ga; B_1} \le C\big(\|v \|_{0,0;B_2} + \hat{r}^2 \|f_1\|'_{0,\ga; B_2}\big).
\end{equation}
The definition of $M$ (see \eqref{def-M}) and assumption \eqref{con-g=0} imply
\begin{align}\label{est-f1ga}
 &\|f_1\|'_{0,\ga; B_2} \le C M \min(r_0^{ \ga -1},r_0^{\gb-2}),\\
 &\|g_1\|'_{1,\ga; B_2} +  \|g_2\|'_{1,\ga; B_2} \le C M \min(r_0^{ \ga },r_0^{\gb-1}).\label{est-g1g2ga}
\end{align}
Estimates  \eqref{est-vinterior}--\eqref{est-f1ga}  and  \eqref{est-vdecay} give rise to
\begin{equation}\label{est-v2ga}
\|v \|'_{2,\ga; B_1}  \le  CM  \min(r_0^{1+\ga},r_0^\gb ).
\end{equation}

For {\it Case {\rm 2}} and  {\it Case {\rm 3}}, we use the Schauder boundary estimates.
Let
\begin{align*}
\hat{r} =  \sin \frac{\omega}{4}, \quad
B_1^+ = B_{\hat{r}}(\by^0) \cap \Pi^R,\quad
B_2^+ = B_{2\hat{r}} (\by^0)\cap \Pi^R,\quad
T= B_{2\hat{r}} (\by^0)\cap \partial \Pi^R.
\end{align*}
Similar to the arguments in {\it Case {\rm 1}},
the Schauder boundary estimates (\textit{cf.} Theorem 6.26 in \cite{gt}),
together with \eqref{est-vdecay} and \eqref{est-f1ga}--\eqref{est-g1g2ga},
lead to
\begin{align}\label{est-vbdry}
\|v \|'_{2,\ga; B_1^+} \le C\Big(\|v \|_{0,0;B_2^+} +\hat{r}\sum_{i=1,2} \|g_i\|'_{1,\ga;T}+ \hat{r}^2 \|f_1\|'_{0,\ga; B_2^+}\Big)
\le C M \min(r_0^{ \ga },r_0^{\gb-1}).
\end{align}
Note that, by the standard scaling argument,
$$
\|v\|_{2,\ga;(-\gb);B_1}^{(-1-\ga)}\leq\frac{1}{\min(r_0^{ \ga },r_0^{\gb-1})} \|v \|'_{2,\ga; B_1},
$$
or
$$
\|v\|_{2,\ga;(-\gb);B_1^+}^{(-1-\ga)}\leq\frac{1}{\min(r_0^{ \ga },r_0^{\gb-1})} \|v \|'_{2,\ga; B_1^+}.
$$
Therefore, estimate \eqref{est-vtrunc} follows by combining the interior estimates \eqref{est-v2ga} for {\it Case {\rm 1}}
with the boundary estimates \eqref{est-vbdry} for {\it Cases {\rm 2}--{\rm 3}}.
This completes the proof.
\end{proof}

Now we are ready to prove Theorem \ref{lem-linear}.

\smallskip
\noindent
{\it Proof of Theorem {\rm \ref{lem-linear}}}.
$\,\,$ For each $R >4$, by Lemma \ref{lem1}, there exists a unique solution $v_R$ for {Problem \ref{prob-trunc}}
satisfying estimate \eqref{est-vtrunc}.
Therefore, by the Ascoli-Azela theorem, we can choose a sequence $R_k \to \infty$ such that the corresponding sequence
of solutions $\{v_{R_k}\}$ converges to a function $v$ in each $C^{2,\ga';(-\gb)}_{(-1-\ga')}(\Pi^{\frac{R_k}{2}})$ for $0 < \ga' <\ga$.
Hence, estimate \eqref{est-vtrunc} indicates that $v \in C^{2,\ga;(-\gb)}_{(-1-\ga)}(\Pi)$ and satisfies estimate \eqref{est-v}.
Clearly, $v$ is a solution of {Problem \ref{prob-linear}} in $\Pi$.

To show the uniqueness of solutions of {Problem \ref{prob-linear}},
we need to prove that, if $v\in  C^{2,\ga;(-\gb)}_{(-1-\ga)}(\Pi) $ and solves
\begin{align*}
\sum_{i,j=1}^{3}a^0_{ij} v_{x_ix_j} = 0  \qquad \text{ in $\Pi$}
\end{align*}
with boundary conditions:
\begin{align*}
v_{x_2}|_{\cw_0} = 0,\quad 	\DD v\cdot \bmu|_{\cs_0} = 0,
\quad v|_{\cE_0} = 0,
\end{align*}
then $v=0$.
We now construct a barrier function and use the comparison principle to achieve this.
It is based on the observation that the solution is uniformly bounded in the $x_3$--direction
and sublinearly grows
in the $(x_1,x_2)$--directions.

For $\gb' \in (\gb,1)$, set
\begin{align*}
v_3 =   \bar{r}^{\gb'} \sin ((\gb' + \tau_2) \bar{\theta} + \frac{\pi}{2} + \tau_2),
\end{align*}
where $\tau_2 >0$ is suitably small.
From estimates \eqref{est-v2domain}--\eqref{est-v2s0}, we have
\begin{align}
&\sum_{i,j=1}^{3}a^0_{ij} (v_3)_{y_iy_j} \le{} -c_1 r^{\gb'-2} \qquad \text{ in $\Pi$}, \label{est-v3domain}\\
&(v_3)_{y_2}|_{\cw_0} \le {}  -c_2 r^{\gb'-1},\\
&\DD v_3\cdot \bmu|_{\cs_0} \ge {}  c_3 r^{\gb'-1}.
\end{align}
Let $v_4 = |\by|^{\gb'}$. It is easy to see the following estimates:
\begin{align}
|\DD_{y_i} v_4 | \le C |\by|^{\gb'-1}, \qquad |\DD^2_{y_iy_j} v_4 | \le C |\by|^{\gb'-2} . \label{est-v4D}
\end{align}
Given any $\tau >0$, define
\begin{align*}
v_5 =   \tau(C_1 v_3 + v_4).
\end{align*}
With estimates \eqref{est-v3domain}--\eqref{est-v4D} and the fact that $r \le |\by|$,
we know that $v_5$ satisfies the following properties:
\begin{align*}
&\sum_{i,j=1}^{3}a^0_{ij} (v_5)_{x_ix_j} \le  0 \qquad \text{ in $\Pi$}, \\
&(v_5)_{x_2}|_{\cw_0} \le 0, \quad
\DD v_5\cdot \bmu|_{\cs_0} \ge0,\quad   v_5|_{\cE_0} \ge 0,
\end{align*}
provided that $C_1$ is suitably large.
We know
\begin{align*}
v_5(\by) \ge \tau v_4 \ge  \tau |\by|^{\gb' - \gb} r^\gb,
\end{align*}
so that $v\in  C^{2,\ga;(-\gb)}_{(-1-\ga)}(\Pi) $ implies that
there exists $C_2 >0$ such that
\begin{align*}
|v(\by)|\le C_2 r^\gb.
\end{align*}
Since $\gb' >\gb$, there exists $R_0 >0$ depending on $(\tau, C_2)$ such that, when $R>R_0$,
\begin{align*}
v_5 (\by)> |v(\by)| \qquad \text{ for $|\by| =R$}.
\end{align*}
Thus, by the comparison principle, we have
\begin{align*}
|v(\by)|  \le  v_5 (\by) \qquad \mbox{for $ \by  \in \Pi \cap B_R(\bO)$}.
\end{align*}
Since $R$ can be arbitrarily large, the above inequality holds for all $\by  \in \Pi $.
Letting $\tau \to 0$, we conclude that $v\equiv 0$, since $v_5\rightarrow0$ everywhere in $\Pi$ as $\tau\rightarrow0$.
This completes the proof.

\smallskip
\section{$\,$ Fixed Point of the Iteration Map: Proof of Theorem 2.1}\label{sec-6}
In order to apply Theorem \ref{lem-linear}, we need to verify the assumptions in Theorem \ref{lem-linear}
for the weak transonic shock problem.
First, we verify the uniform ellipticity of equation \eqref{eqn-vflinear}, which is summarized in the following proposition:

\begin{prop}\label{prop:5.1a}
If $\varphi^+_0$ is a uniform subsonic solution, then
there exists $\lambda >0$ depending on $\varphi_0^+$ such that
\begin{align}\label{5.1a}
\lambda |\boldsymbol{\xi}|^2 \le \sum_{i,j=1}^{3}a_{ij}^0 \xi_i\xi_j \le \frac{1}{\lambda}  |\boldsymbol{\xi}|^2
\qquad\mbox{for any $\boldsymbol{\xi} \in \R^3$},
\end{align}	
that is, equation \eqref{eqn-shocklinear} is uniformly elliptic.
\end{prop}
\begin{proof}
$\,$ Since $\varphi^+_0$ is a weak transonic solution, $\varphi_{0}^+$ is a uniform subsonic solution.
By the definition, $a^0_{ij}=a_{ij}(U^+_0)=c^2(|\DD\varphi_0^+|^2)\delta_{ij}-\partial_{x_i}\varphi_0^+\partial_{x_j}\varphi_0^+$.
Hence, claim \eqref{5.1a} follows immediately from the fact that the background solution $\varphi_{0}^+$
is a uniform subsonic solution.
\end{proof}

Next, we verify the obliqueness of the boundary condition on the shock and the non-obliqueness of the boundary condition
at the wedge edge.
More precisely, for a piecewise constant weak transonic flow $(\vf_0^-, \vf_0^+)$,
we need to check the direction of $\bmu$ defined by \eqref{def-mu}, which is described as follows:
\begin{prop}\label{prop:5.2a}
$\,$ If $(\varphi_0^-,\varphi_0^+ )$ is a weak transonic shock solution, then
\begin{align*}
\mu_1 >0, \qquad \mu_2 >0, \qquad \bmu\cdot\bn_{\rm sh}>0,
\end{align*}
where $\bn_{\rm sh}$ is the outer unit normal of the shock surface $\cs_0$.
\end{prop}

\begin{proof}
$\,$ In equation \eqref{def-H}, we fix the incoming flow $\bu= U_0^- $ given as in \eqref{def-u0-}.
By the Rankine-Hugoniot condition \eqref{con-RH4} and the fact that the $x_3$--axis is the wedge edge (\emph{i.e.}, the $x_3$--axis lies in the shock plane),
we see that $v_3=q_0\cos\theta_{\rm i}$.
Then $v_2$ can be considered as a smooth function of $v_1$ by solving \eqref{def-H} for $v_2$.
There are two solutions of equation $v_2(v_1)=0$ (owing to the choice of the coordinates such that the downstream subsonic flow
is parallel to the $(x_1,x_3)$--plane in \S \ref{subsec:background}); also see Fig. 3.1.
We denote the two solutions by $v_{\rm s}$ and $v_{\rm w}$, which correspond to the strong and weak transonic shock solutions, respectively.
Using the convexity of the two-dimensional shock polar for the potential flow (see {\it e.g.} \cite[Lemma 7.3.2, page 249]{CF2}),
we have the following properties:
\begin{align*}
&v_{\rm s} < v_{\rm w},\quad v_2'(v_{\rm s}) >0,\quad v_2'(v_{\rm w}) <0,\\
& v_2(v_1) >0\qquad \mbox{for any $v_2 \in (v_{\rm s} ,v_{\rm w})$}.
\end{align*}
	
Take the weak transonic flow by letting $v_0= v_{\rm w}$.
Then we have the downstream flow
$U_0^+ =(v_0 ,0, q_0\cos\theta_{\rm i})$ in \eqref{def-u0+}.
Differentiating
\begin{align*}
H(U^-_0, v_1, v_2(v_1), q_0\cos\theta_{\rm i}) = 0
\end{align*}
with respect to $v_1$ and letting $v_1= v_{\rm w}$ imply
\begin{align*}
\big(H_{v_1}+ H_{v_2}v_2'\big)(U^-_0, U^+_0)=0.
\end{align*}
	
By the definition that $\bmu=H_{\bv}(U^-_0, U^+_0)$,
$v_2'(v_{\rm w}) <0$ implies $\mu_1 \mu_2 >0$.
A direct computation shows
\begin{align*}
\mu_2 = q_0 \sin\theta_{\rm i} \sin \theta_{\rm w} \big(\rho(|U_0^-|^2) +\rho(|U_0^+|^2)\big) >0,
\end{align*}
which yields that $\mu_1 >0$.
	
Next, we show that $\bmu\cdot\bn_{\rm sh}>0$.
Notice that
\begin{align*}
\bn_{\rm sh} = \frac{1}{\sqrt{1+\sigma^2}}(-\sigma, 1,0),
\end{align*}
where $\sigma$ is the slope given by \eqref{exp-k0}.
Thus, it suffices to prove
\begin{align*}
-\mu_1 \sigma + \mu_2 >0.
\end{align*}
To simplify the notation, denote
\begin{align*}
U_0^-&=(u_1, u_2,u_3), \quad U_0^+=(v_1, 0,u_3),\\
\rho^-&=\rho(|U_0^-|^2), \quad\rho^+=\rho(|U_0^+|^2), \quad c^+= c(|U_0^+|^2).
\end{align*}
Then we have
\begin{align*}
&\sigma=-\frac{u_1 -v_1}{u_2},\\
&\mu_1=\rho^+\Big(\frac{v_1^2}{(c^+)^2}-1\Big)(u_1 -v_1)-\rho^-u_1+\rho^+v_1,\\
&\mu_2=-u_2(\rho^- + \rho^+),
\end{align*}
so that
$$
-\mu_1 \sigma + \mu_2
 = -\frac{\rho^+}{u_2}\Big(\big(1-\frac{v_1^2}{(c^+)^2}\big)(u_1 -v_1)^2 +u_2^2\Big).
$$
By the subsonicity of $U_0^+$, we see that $v_1^2 < (c^+)^2 $.
Together with the fact that $u_2 <0$, we conclude that $-\mu_1 \sigma + \mu_2>0 $.
\end{proof}

Therefore, based on Propositions \ref{prop:5.1a}--\ref{prop:5.2a},
the assumptions in Theorem \ref{lem-linear} hold for the solutions that are small perturbations of the weak transonic plane shock.
Using Theorem \ref{lem-linear} and following the iteration scheme introduced in \S \ref{sec-iteration},
given $(\gd \hs, \gd \vf)\in\mathcal{K}$, we solve {Problem \ref{prob-linear}} uniquely with
$v=\gd \tvf$, $f_1=f^{\hs}$, $g_1=g_{\rm w}^{\hs}$, $g_2=g_{\rm s}^{\hs}$, and $g_3=g_{\rm e}^{\hs}$,
where $\mathcal{K} $ is defined by \eqref{def-K1}--\eqref{def-K2}, and the expressions of $f^{\hs}_0$, $g_{\rm w}^{\hs}$,
$g_{\rm s}^{\hs}$, and $g_{\rm e}^{\hs}$ are given by \eqref{def-f}, \eqref{def-gw}, \eqref{def-gs}, and \eqref{def-ge}, respectively.
Then we are going to show Theorem \ref{thm-main} by establishing the contraction of $\ft$ defined by \eqref{def:T},
where $\delta\tilde{s}$ is given by \eqref{eqn-updateshock}.

\medskip
\smallskip
\noindent
{\it Proof of Theorem {\rm \ref{thm-main}} $($Main Theorem$)$}. $\,\,$ The proof is divided into four steps.

\smallskip
1. $\,$ We first show that the map is well-defined: $\ft$ is a map from $\mathcal{K}$ to itself.

Based on Propositions \ref{prop:5.1a}--\ref{prop:5.2a},
the conditions in Theorem \ref{lem-linear} are satisfied so that we can uniquely solve $\delta\tilde{\varphi}=v$
with the following estimate:
\begin{align}\label{est-dvftilde}
\|\gd \tvf\|_{2,\ga;(-\gb);\Pi}^{(-1-\ga)}
\le C \Big(\|f^{\hs}\|_{0,\ga;(2-\gb);\Pi}^{(1-\ga)} + \|g_{\rm w}^{\hs}\|_{1,\ga;(1-\gb);\cd_0}^{(-\ga)}
+ \|g_{\rm s}^{\hs}\|_{1,\ga;(1-\gb);\cf_0}^{(-\ga)} +  \|g_{\rm e}^{\hs}\|_{1,\ga;\R}\Big).
\end{align}	

First, we need to estimate the right-hand side of \eqref{est-dvftilde} carefully.
Based on the definition of the coordinate transformation \eqref{def-transy2y3} and \eqref{def-transy1},
a straightforward calculation gives
\begin{align}
&\frac{\partial \by}{\partial {x_1}}=(\frac{1 +(\gd \ls)_{y_2}w_{x_1} }{1+ (\gd \ls)_{y_1}}, -w_{x_1}, 0),\label{exp-dydx1}\\
&\frac{\partial \by}{\partial {x_2}}=(-\frac{(\gd \ls)_{y_2} }{1+ (\gd \ls)_{y_1}}, 1,0),\\
&\frac{\partial \by}{\partial {x_3}}=(\frac{(\gd \ls)_{y_2}w_{x_3}-(\gd \ls)_{y_3} }{1+ (\gd \ls)_{y_1}}, -w_{x_3}, 1).
\label{exp-dydx4}
\end{align}
Thus, we have the following estimate:
\begin{align*}
\|\frac{\partial y_i}{\partial x_j}-\delta_{ij}\|_{1,\alpha;(1-\beta);\Pi}^{(-\alpha)}
+ \|\frac{\partial^2 y_i}{\partial x_k \partial x_m}\|_{0,\alpha;(2-\beta),\Pi}^{(1-\alpha)}
&\leq C\Big(\|w\|_{2,\alpha;(-\beta);\cd_{e_1}}^{(-1-\alpha)}+\|\gd \hs\|_{2,\alpha;(-\beta);\cf_0}^{(-1-\alpha)}\Big)
\end{align*}
for any $i, j, k, m=1, 2, 3$.

Then, by the definition of $f^{\hs}$ (see \eqref{def-atilde}--\eqref{def-btilde} and \eqref{def-f}), we have
 \begin{align*}
\|f^{\hs}\|_{0,\ga;(2-\gb);\Pi}^{(1-\ga)}
&\le{} \|\gd \vf\|_{2,\ga;(-\gb);\Pi}^{(-1-\ga)}\Big( \sum_{i,j=1}^3\|a_{ij}^0 - \tilde{a}_{ij}^{\hs}\|_{1,\alpha;(1-\beta);\Pi}^{(-\alpha)} +\sum_{i=1}^3\|\tilde{b}_{i}^{\hs}\|_{0,\alpha;(2-\beta),\Pi}^{(1-\alpha)} \Big) \\
& \le{} C\|\gd \vf\|_{2,\ga;(-\gb);\Pi}^{(-1-\ga)}
  \Big(\|w\|_{2,\alpha;(-\beta);\cd_{e_1}}^{(-1-\alpha)}+\|\gd \hs\|_{2,\alpha;(-\beta);\cd_0}^{(-1-\alpha)}+\|\gd \vf\|_{2,\ga;(-\gb);\Pi}^{(-1-\ga)}\Big)\\
&\le{} CC_0^2\ve^2,
 \end{align*}
provided that $C_0\ve < \lambda_0$, where $\lambda_0>0$ is a fixed  constant depending on $\vf_0^+$, and $C$ depends on $\lambda_0$.
Similarly, by definitions \eqref{def-gw}, \eqref{def-gstilde}, \eqref{def-gs}, and \eqref{def-ge}, we have the following estimates:
 \begin{align*}
&\big\|g_{\rm w}^{\hs}|_{\cw_0}\big\|_{1,\ga;(1-\gb);\cd_0}^{(-\ga)}\le C \|w\|_{2,\alpha;(-\beta);\cd_{e_1}}^{(-1-\alpha)} \le C\ve,\\[1mm]
& \big\|\tilde{g}_{\rm s}^{\hs}|_{\cs_0}\big\|_{1,\ga;(1-\gb);\cf_0}^{(-\ga)}
 \le C  \Big(\|\gd \vf^-_{\hs}\|_{2,\ga;(-\gb);\Pi^-}^{(-1-\ga)}+(\|\gd \vf\|_{2,\ga;(-\gb);\Pi}^{(-1-\ga)})^2\Big),\\[1mm]
& \big\|g_{\rm s}^{\hs}|_{\cs_0}\big\|_{1,\ga;(1-\gb);\cf_0}^{(-\ga)}
  \le C \|\gd \vf\|_{2,\ga;(-\gb);\Pi}^{(-1-\ga)}\|I - J_{\bx} \by\|_{1,\ga;(1-\gb);\Pi}^{(-\ga)}
   +  \big\|\tilde{g}_{\rm s}^{\hs}|_{\cs_0}\big\|_{1,\ga;(1-\gb);\cf_0}^{(-\ga)}
    \le C(\ve + C_0^2\ve^2),\\
&\big\|g_{\rm e}^{\hs}\big\|_{1,\ga;\R}
 \le C\Big( \|\gd \vf^-_{\hs}\|_{2,\ga;(-\gb);\Pi^-}^{(-1-\ga)} +\|w\|_{2,\alpha;(-\beta);\cd_{e_1}}^{(-1-\alpha)}+  \|e_1\|_{1,\ga;\R}\Big)\le C\ve.
\end{align*}
Therefore, we obtain
\begin{align}\label{est-dphitilde}
 \|\gd \tvf\|_{2,\ga;(-\gb);\Pi}^{(-1-\ga)} \le C\ve\big(1 + C_0^2 \ve\big) \le C_0 \ve,
\end{align}
by choosing $C_0 > 2C$ and $\ve < \frac{1}{C_0^2}$.

Next, we consider the estimate for $\gd \ts$ obtained via \eqref{eqn-updateshock}. Rewrite \eqref{eqn-updateshock} into
 \begin{align*}
\big((\vf^- -\vf^-_0) +(\vf^-_0-\vf^+_0)\big)(s_0+ \gd \ts, y_2+w(s_0+ \gd \ts,y_3),y_3)
 = \gd \tvf(\frac{y_2}{\sigma},y_2,y_3).
\end{align*}
Since $\big(\vf^-_0-\vf^+_0\big)|_{\cs_0}=0$, the equality above gives
\begin{align}\label{exp-dstilde}
\gd \ts (y_2,y_3)
=&\,\frac{1}{\sigma}w(s_0+ \gd \ts,y_3)    \nonumber\\
&\, + \frac{\gd \tvf(\frac{1}{\sigma}y_2,y_2,y_3)
- \big(\vf^- -\vf^-_0\big)(s_0+ \gd \ts, y_2+w(s_0+ \gd \ts,y_3),y_3)}{q_0 \sin\theta_{\rm i} \cos \theta_{\rm w} - u_0},
\end{align}
which leads to the following estimate:
\begin{align}
\|\gd \ts\|_{2,\ga;(-\gb);\cf_0}^{(-1-\ga)}
\le C\Big( \|\gd \tvf\|_{2,\ga;(-\gb);\Pi}^{(-1-\ga)} +\|\vf^- - \vf^-_0\|_{2,\ga;(-\gb);\gO^-_e}^{(-1-\ga)}+ \|w\|_{2,\ga;(-\gb);\cd_{e_1}}^{(-1-\ga)}\Big)
\le C_0 \ve. \label{est-dstilde}
\end{align}

Therefore, we have shown that $\ft$ is a map from $\mathcal{K}$ to itself.
Finally, we remark that estimate \eqref{est-dstilde} for $\gd \ts$ also
guarantees that the updated shock surface in the $\by$--coordinates stays in $\Pi^-$.

\smallskip
2. $\,$ In this step, we show the contraction of  $\ft$.

Given two pairs $(\gd \hs^{(i)}, \gd \vf^{(i)}) \in \mathcal{K}$, let $(\gd \ts^{(i)}, \gd \tvf^{(i)})=\ft(\gd \hs^{(i)}, \gd \vf^{(i)})$ for $i=1$, $2$.
By the definition of $\ft$, we know that $v= \gd \tvf^{(1)} - \gd \tvf^{(2)}$ solves {Problem \ref{prob-linear}} with
\begin{align}
f_1 & = f^{\hs^{(1)}} (\by, \DD (\gd \vf^{(1)}), \DD^2 ( \gd \vf^{(1)}))- f^{\hs^{(2)}}(\by,\DD (\gd \vf^{(2)}), \DD^2 (\gd \vf^{(2)})),\label{def-f112}\\
g_1& =g_{\rm w}^{\hs^{(1)}}(\by, \DD (\gd \vf^{(1)})) - g_{\rm w}^{\hs^{(2)}}(\by,\DD (\gd \vf^{(2)})),\label{def-g112}\\
g_2& =g_{\rm s}^{\hs^{(1)}} (\by, \DD (\gd \vf^{(1)}))- g_{\rm s}^{\hs^{(2)}}(\by,\DD (\gd \vf^{(2)})),\label{def-g212}\\
g_3& =g_{\rm e}^{\hs^{(1)}} - g_{\rm e}^{\hs^{(2)}}.\nonumber
\end{align}

Since
\begin{align*}
g_{\rm e}^{\hs^{(i)}} &= \big(\vf^- -\vf^+_0\big)(\gd \hs^{(i)}(0,y_3), w(\gd \hs^{(i)}(0,y_3)),y_3)
= \big(\vf^- -\vf^+_0\big)(e_1(y_3), e_2(y_3),y_3),
\end{align*}
we have
\begin{align}\label{5.14}
g_3& =g_{\rm e}^{\hs^{(1)}} - g_{\rm e}^{\hs^{(2)}}=0.
\end{align}

Denote the coordinate transformation related to  $\gd \hs^{(i)}$ by $\by^{(i)}(\bx)$,
and set
\begin{align*}
\lw^{(i)} (\by):= w(y_1 + \gd \ls^{(i)}(\by), y_3),\qquad\, \lw^{(i)}_{x_j} (\by):= w_{x_j}(y_1 + \gd \ls^{(i)}(\by), y_3).
\end{align*}
It is direct to see
\begin{align}
\lw^{(1)}-\lw^{(2)}&=\int_0^1w_{x_1}(y_1+\tau \delta\ls^{(1)}(\by)+(1-\tau)\delta\ls^{(2)}(\by)) {\rm d}\tau\, (\delta\ls^{(1)}-\delta\ls^{(2)}), \label{exp-wbar12}\\
\lw_{x_j}^{(1)}-\lw_{x_j}^{(2)}&=\int_0^1w_{x_1x_j}(y_1+\tau \delta\ls^{(1)}(\by)+(1-\tau)\delta\ls^{(2)}(\by)) {\rm d}\tau\, (\delta\ls^{(1)}-\delta\ls^{(2)}).\label{exp-wbarxj12}
\end{align}
Thus, with assumption \eqref{2.17}, we have
\begin{align}
&\|\lw^{(1)} - \lw^{(2)}\|_{2,\ga;(-\gb);\cd_0}^{(-1-\ga)} + \sum_{j=1}^3 \|\lw_{x_j}^{(1)}-\lw_{x_j}^{(2)}\|_{1,\ga;(1-\gb);\cd_0}^{(-\ga)}\nonumber\\
&\leq C\|w_{x_1}\|_{2,\ga;(-\gb);\cd_{e_1}}^{(-1-\ga)}\|\delta\hs^{(1)} - \delta\hs^{(2)}\|_{2,\ga;(-\gb);\cf_0}^{(-1-\ga)} \nonumber\\[1mm]
&\le C\ve \|\delta\hs^{(1)} - \delta\hs^{(2)}\|_{2,\ga;(-\gb);\cf_0}^{(-1-\ga)}. \label{est-w12}
\end{align}
Then expressions  \eqref{exp-dydx1}--\eqref{exp-dydx4} and estimate \eqref{est-w12} lead to
\begin{align}
&\big\|\frac{\partial y_i^{(1)}}{\partial x_j}-\frac{\partial y_i^{(2)}}{\partial x_j}\big\|_{1,\alpha;(1-\beta);\Pi}^{(-\alpha)}
\leq C\ve \|\delta\hs^{(1)} - \delta\hs^{(2)}\|_{2,\ga;(-\gb);\cf_0}^{(-1-\ga)},\\
&\big\|\frac{\partial^2 y_i^{(1)}}{\partial x_j \partial x_m}-\frac{\partial^2 y_i^{(2)}}{\partial x_j \partial x_m}\big\|_{0,\alpha;(2-\beta);\Pi}^{(1-\alpha)}
\leq C\ve \|\delta\hs^{(1)} - \delta\hs^{(2)}\|_{2,\ga;(-\gb);\cf_0}^{(-1-\ga)}.\label{est-y12xjxm}
\end{align}

Since the definition of $f_1$ in \eqref{def-f112} involves \eqref{def-atilde}--\eqref{def-btilde} and \eqref{def-f},
estimates \eqref{est-w12}--\eqref{est-y12xjxm} imply
\begin{align}
\|f_1 \|_{0,\alpha;(2-\beta);\Pi}^{(1-\alpha)}
&\leq C\ve \Big(\|\delta\hs^{(1)} - \delta\hs^{(2)}\|_{2,\ga;(-\gb);\cf_0}^{(-1-\ga)} + \|\gd \vf^{(1)} - \gd \vf^{(2)}\|_{2,\ga;(-\gb);\Pi}^{(-1-\ga)}\Big).
\label{est-f112}
\end{align}

In the same manner, we can obtain the following estimates for $g_1$ in \eqref{def-g112} and $g_2$ in \eqref{def-g212}:
\begin{align}
&\big\|g_1|_{\cw_0}\big\|_{1,\ga;(1-\gb);\cd_0}^{(-\ga)}
 +\big\|g_2|_{\cs_0}\big\|_{1,\ga;(1-\gb);\cf_0}^{(-\ga)}\\
& \le C\ve\Big(\|\delta\hs^{(1)} - \delta\hs^{(2)}\|_{2,\ga;(-\gb);\cf_0}^{(-1-\ga)} + \|\gd \vf^{(1)} - \gd \vf^{(2)}\|_{2,\ga;(-\gb);\Pi}^{(-1-\ga)}\Big).
\label{est-g1g2}
\end{align}
Therefore, by Theorem \ref{lem-linear}, estimates \eqref{5.14}, \eqref{est-f112}, and \eqref{est-g1g2} imply
\begin{align}
\|\gd \tvf^{(1)} - \gd \tvf^{(2)} \|_{2,\alpha;(-\beta);\Pi}^{(-1-\alpha)}
&\leq C\ve \Big(\|\delta\hs^{(1)} - \delta\hs^{(2)}\|_{2,\ga;(-\gb);\cf_0}^{(-1-\ga)} + \|\gd \vf^{(1)} - \gd \vf^{(2)}\|_{2,\ga;(-\gb);\Pi}^{(-1-\ga)}\Big).
\label{est-phitilde12}
\end{align}

We now estimate the difference between the two updated shocks. By identity \eqref{exp-dstilde}, we have
 \begin{align*}
&\gd \ts^{(1)}  - \gd \ts^{(2)} \\
&= \frac{1}{\sigma}w(s_0+ \gd \ts^{(1)},y_3) - \frac{1}{\sigma}w(s_0+ \gd \ts^{(2)},y_3)\\
&\quad+ \frac{1}{q_0 \sin\theta_{\rm i} \cos \theta_{\rm w} - u_0}\Big(\gd \tvf^{(1)}(\frac{1}{\sigma}y_2,y_2,y_3) -\gd \tvf^{(2)}(\frac{1}{\sigma}y_2,y_2,y_3) \\
&\quad\qquad\qquad\qquad\qquad\qquad\, - \big(\vf^- -\vf^-_0\big)(s_0+ \gd \ts^{(1)}, y_2+w(s_0+ \gd \ts^{(1)},y_3),y_3)\\
&\quad\qquad\qquad\qquad\qquad\qquad\,  + \big(\vf^- -\vf^-_0\big)(s_0+ \gd \ts^{(2)}, y_2+w(s_0+ \gd \ts^{(2)},y_3),y_3)\Big).
\end{align*}
Following the same approach from \eqref{exp-wbar12} to \eqref{est-w12}, we can write $ \gd \ts^{(1)}  - \gd \ts^{(2)} $
in an integral form and use assumption \eqref{2.17} to obtain the following estimate:
\begin{align*}
&	\|\gd \ts^{(1)}  - \gd \ts^{(2)} \|_{2,\ga;(-\gb);\cf_0}^{(-1-\ga)} \\
&\leq C\Big(\|w_{x_1}\|_{2,\ga;(-\gb);\cd_{e_1}}^{(-1-\ga)} + \|\vf^- - \vf^-_0\|_{3,\ga;(-\gb);\gO^-_e}^{(-2-\ga)}\Big)\|\gd \ts^{(1)}
  - \gd \ts^{(2)} \|_{2,\ga;(-\gb);\cf_0}^{(-1-\ga)} \\[0.5mm]
 &\quad
   + C\|\gd \tvf^{(1)} - \gd \tvf^{(2)}\|_{2,\ga;(-\gb);\Pi}^{(-1-\ga)}\\
&\leq C \ve  \|\gd \ts^{(1)}  - \gd \ts^{(2)} \|_{2,\ga;(-\gb);\cf_0}^{(-1-\ga)} + C \|\gd \tvf^{(1)} - \gd \tvf^{(2)}\|_{2,\ga;(-\gb);\Pi}^{(-1-\ga)}.
\end{align*}
Choosing $\ve$ sufficiently small so that $C\ve <\frac{1}{2}$, we conclude
\begin{align}
	\|\gd \ts^{(1)}  - \gd \ts^{(2)} \|_{2,\ga;(-\gb);\cf_0}^{(-1-\ga)} \le  C \|\gd \tvf^{(1)} - \gd \tvf^{(2)}\|_{2,\ga;(-\gb);\Pi}^{(-1-\ga)}. \label{est-dstilde12}
\end{align}
This, together with \eqref{est-f112}, gives rise to
\begin{align}\label{est-stilde122}
	\|\gd \ts^{(1)}  - \gd \ts^{(2)} \|_{2,\ga;(-\gb);\cf_0}^{(-1-\ga)}
	&\leq C\ve \Big(\|\delta\hs^{(1)} - \delta\hs^{(2)}\|_{2,\ga;(-\gb);\cf_0}^{(-1-\ga)} + \|\gd \vf^{(1)} - \gd \vf^{(2)}\|_{2,\ga;(-\gb);\Pi}^{(-1-\ga)}\Big).
\end{align}
Estimates \eqref{est-phitilde12} and \eqref{est-stilde122} imply
\begin{align}
&\|\gd \ts^{(1)}  - \gd \ts^{(2)} \|_{2,\ga;(-\gb);\cf_0}^{(-1-\ga)}
+  \|\gd \tvf^{(1)} - \gd \tvf^{(2)}\|_{2,\ga;(-\gb);\Pi}^{(-1-\ga)}\nonumber\\
&\leq C \ve \Big(\|\delta\hs^{(1)} - \delta\hs^{(2)}\|_{2,\ga;(-\gb);\cf_0}^{(-1-\ga)} + \|\gd \vf^{(1)} - \gd \vf^{(2)}\|_{2,\ga;(-\gb);\Pi}^{(-1-\ga)}\Big),
\label{est-s-phi-tilde12}
\end{align}
which leads to the contraction of $\ft$, provided that  $\ve$ is sufficiently small.

\smallskip
3. $\,$ Based on \eqref{est-s-phi-tilde12},
by the Banach fixed-point theorem,  there exists a unique fixed point $(\gd s^*, \gd  \varphi^*)$ of $\ft$ in $\mathcal{K}$.
It follows from the definition of the coordinate transformation \eqref{def-transy2y3}--\eqref{def-transy1}
that $\gd s^*$ uniquely determines transformation $\by(\bx)$.
Thus, by \eqref{def-dphiy}, we set $\vf^+(\bx):= \vf^*(\by(\bx))$, which is the unique solution of {Problem \ref{prob}}.

To show estimate \eqref{est-thm}, we set
\begin{align*}
\ve':= \|\vf^- - \vf^-_0\|_{2,\ga;(-\gb);\gO^-_e}^{(-1-\ga)} + \|e_1\|_{1,\ga;\R}+ \|w\|_{2,\ga;(-\gb);\cd_{e_1}}^{(-1-\ga)} .
\end{align*}
By estimates \eqref{est-dphitilde} and \eqref{est-dstilde}, we obtain
 \begin{align}\label{5.25}
 	 \|\gd s^*   \|_{2,\ga;(-\gb);\cf_0}^{(-1-\ga)}
 		+  \|\gd \vf^* \|_{2,\ga;(-\gb);\Pi}^{(-1-\ga)}
 		\leq{}& C \ve',
 \end{align}
which is equivalent to \eqref{est-thm} in the $\bx$-coordinates
when transforming the variables back.

Moreover, the uniqueness follows directly from the contraction of mapping $\mathcal{T}$.

\smallskip
4. $\,$ If $|\vf^--\vf^-_0|+|\DD\vf^--\DD\vf^-_0|+|e_1|+|w|\rightarrow0$ as $x_3\rightarrow\infty$ (or as $x_3\rightarrow-\infty$) pointwise,
then, in the fixed domain $\Pi$, for any fixed $(y_1,y_2)$, $|\vf^--\vf^-_0|+|\DD\vf^--\DD\vf^-_0|\rightarrow0$
as $y_3\rightarrow\infty$ (or as $y_3\rightarrow-\infty$).
Without loss of the generality, it suffices to consider the case that $y_3\rightarrow\infty$,
since the same argument works for the case that $y_3\rightarrow-\infty$ if replacing $y_3+r_n$ by $y_3-r_n$
in the argument below for any sequence $r_n\to \infty$ as $n\to \infty$.

Let $\vf_n(y_1,y_2,y_3)=\vf(y_1,y_2,y_3+r_n)$ on $\Pi$, and $s_n(y_2,y_3)=s(y_2,y_3+r_n)$ on $\mathcal{F}_0$.
Let $\vf^-_n(x_1,x_2,x_3)=\vf(x_1,x_2,x_3+r_n)$ on $\gO^-_{\rm e}$, $e_1^{(n)}(x_3)=e_1(x_3+r_n)$ on $\mathbb{R}$,
and $w_n(x_1,x_3)=w(x_1,x_3+r_n)$ on $\cd_{e_1}$.
Then it follows from \eqref{5.25} that
\begin{align*}
	\| s_n-s_0   \|_{2,\ga;(-\gb);\cf_0}^{(-1-\ga)}
	+  \|\vf_n-\vf_0^+ \|_{2,\ga;(-\gb);\Pi}^{(-1-\ga)}
	\leq{}& C \ve',
\end{align*}
where 	
\begin{align*}
\ve':= \|\vf^-_n - \vf^-_0\|_{2,\ga;(-\gb);\gO^-_{\rm e}}^{(-1-\ga)} + \|e_1^{(n)}\|_{1,\ga;\R}+ \|w_n\|_{2,\ga;(-\gb);\cd_{e_1}}^{(-1-\ga)},
\end{align*}
which is independent of $n$, by the definition of $(\vf^n, e_1^{(n)}, w_n)$ and  the H\"{o}lder norms in \eqref{def:2.27}--\eqref{def-norm}.
Thus, by the compact embedding for the bounded weighted H\"{o}lder norms,
there exist subsequences (still denoted as) $(s_n, \varphi_n)$ and functions $(\bar{s}, \bar{\vf})$ such that,
for $0<\ga' <\ga$,
$$
\|\vf_n-\bar{\vf} \|_{2,\ga';(-\gb);\Pi}^{(-1-\ga')}+\| s_n-\bar{s}   \|_{2,\ga';(-\gb);\cf_0}^{(-1-\ga')}\rightarrow0
\qquad\mbox{as $n\rightarrow\infty$}.
$$
Moreover, because $\vf_n^-\rightarrow\varphi_{0}^-$, $e_1^{(n)}\rightarrow0$, and $w_n\rightarrow0$ pointwise,
going back to the original $\mathbf{x}$--coordinates, it is direct to see that the subsonic solution $\bar{\vf}$
with the shock surface $x_1=\bar{s}(x_2,x_3)$ is a solution of Problem \ref{prob} with the incoming flow $\vf_0^-$,
the wedge surface $x_2=0$, and the wedge edge $(x_1,x_2)=(0,0)$.
In addition, $\bar{s}$ and $\bar{\vf}$ satisfy
\begin{align}\label{5.26}
	 \|\bar{\vf}-\vf_0^+ \|_{2,\ga;(-\gb);\Omega_s}^{(-1-\ga)}
     + \| \bar{s}-s_0   \|_{2,\ga;(-\gb);\cf_{e_2}}^{(-1-\ga)}
	\leq{}& C \ve'.
\end{align}
Note that $\vf_0^+$ with the shock surface $x_1=s_0(x_2,x_3)$ is another subsonic solution of Problem \ref{prob}
with the incoming flow $\vf_0^-$, the wedge surface $x_2=0$, and the wedge edge $(x_1,x_2)=(0,0)$,
and satisfying estimate \eqref{5.26}.
By the uniqueness of such solutions, we obtain that $\bar{s}=s_0$ and $\bar{\vf}=\vf_0^+$.
Since all the limits of converging subsequences of $\varphi_n$ and $s_n$ tend to the same functions $\varphi^+$ and $s$
as $n\rightarrow\infty$, we have
\[
|\varphi^+-\varphi_0^+|+|\DD\varphi^+-\DD\varphi_0^+|+|s-s_0|\rightarrow0
\]
as $x_3\rightarrow\infty$ pointwise.

This completes the proof of Theorem \ref{thm-main}.

\begin{acknowledgements}
$\,$ The research of Gui-Qiang G. Chen was supported in part by the UK Engineering and Physical Sciences Research Council Award EP/L015811/1, and the Royal Society--Wolfson Research Merit Award  WM090014 (UK).
Jun Chen's research was supported in part by Yichun University Doctoral Start-up Grant 207-3360119008 and NSFC Regional Science Funds 12061080 (China).
The research of Wei Xiang was supported in part by the Research Grants Council of the HKSAR, China
(Project CityU 11303518, Project CityU 11304820, and Project CityU 11300021).
\end{acknowledgements}

%
\section*{Conflict of interest}
The authors declare that they have no conflict of interest.


\begin{thebibliography}{}
\bibitem{BCF}
M. Bae, G.-Q. Chen, and M. Feldman,
Regularity of solution to regular shock reflection for potential flow.
\newblock
\textit{Invent. Math.} \textbf{175} (2009), 505--543.

\bibitem{BCF1}
M. Bae, G.-Q. Chen, and M. Feldman,
Prandtl-Meyer reflection for supersonic flow past a solid ramp.
\newblock
\textit{Quarterly Appl. Math.} \textbf{71} (2013), 583--600.

\bibitem{BCF2}
M. Bae, G.-Q. Chen, and M. Feldman,
\textit{Prandtl-Meyer Reflection Configurations,
Transonic Shocks, and Free Boundary Problems}.
Research Monograph,
Memoirs of the American Mathematical Society, AMS: Providence, 2021 (to appear).
[Preprint arXiv:1901.05916, 2019]


\bibitem{CCF} G.-Q. Chen, J. Chen, and M. Feldman,
Transonic flows with shocks past curved wedges for the full Euler equations.
\textit{Discrete Conti. Dyn. Syst.} \textbf{36} (2016), 4179--4211.

\bibitem{CCF1} G.-Q. Chen, J. Chen, and M. Feldman,
Stability and asymptotic behavior of tranosnic flows past wedges for the full Euler equations.
\textit{Interfaces Free Bound.} \textbf{19} (2017), 591--626.

\bibitem{CF} G.-Q. Chen and B. Fang,
Stability of transonic shocks in steady supersonic flow past multidimensional wedge.
\textit{Adv. Math.} \textbf{314} (2017), 493--539.

\bibitem{CF1} G.-Q. Chen and M. Feldman,
Global solutions of shock reflection by large angle wedge for potential flow.
\textit{Ann. Math.} (2), \textbf{71} (2010), 1067--1182.

\bibitem{CF2} G.-Q. Chen and M. Feldman,
\textit{The Mathematics of Shock Reflection-Diffraction and von Neumann's Conjectures}.
Research Monograph, Annals of Mathematics Studies, 197,
Princeton University Press, 2018.

\bibitem{CFX} G.-Q. Chen, M. Feldman, and W. Xiang,
Convexity of self-similar transonic shocks and free boundaries for potential flow.
\textit{Arch. Ration. Mech. Anal.} \textbf{238} (2020), 47--124.

\bibitem{CHWX} G.-Q. Chen, F. Huang, T.-Y. Wang, and W. Xiang,
Steady Euler flows with large vorticity and characteriestic discontinuities in arbitrary infinitely long nozzles.
\textit{Adv. Math.}  \textbf{346} (2019), 946--1008.


\bibitem{ChenLi} G.-Q. Chen and T. Li,
{Well-posedness for two-dimensional steady supersonic Euler flows past a Lipschitz wedge}.
\textit{J. Diff. Eqs.} \textbf{244} (2008), 1521--1550.


\bibitem{ChenZhangZhu} G.-Q. Chen, Y. Zhang, and D. Zhu, Existence
and stability of supersonic Euler flows past Lipschitz wedges.
\textit{Arch. Ration. Mech. Anal.} \textbf{181} (2006), 261--310.


\bibitem{Chen2} S.-X. Chen, Asymptotic behavior of supersonic flow
past a convex combined wedge. \textit{China Ann. Math.}
\textbf{19B} (1998), 255--264.

\bibitem{Chen3} S.-X. Chen, Global existence of supersonic flow past
a curved convex wedge. \textit{J. Partial Diff. Eqs.} \textbf{11}
(1998), 73--82.

\bibitem{Sxchen} S.-X. Chen, Existence of local solution to supersonic
flow past a three-dimensional wing. \textit{Adv. Appl. Math.}
\textbf{13} (1992), 273--304.

\bibitem{ChenFang} S.-X. Chen and B. Fang, Stability of transonic
shocks in supersonic flow past a wedge. \textit{J. Diff. Eqs.} \textbf{233}
(2007), 105--135.

\bibitem{CY} S.-X. Chen and C. Yi, Global solutions for supersonic flow past a Delta wing.
\textit{SIAM J. Math. Anal.} \textbf{47} (2015), 80--126.

\bibitem{CDX} J. Cheng, L. Du, and W. Xiang,
Incompressible R\'{e}thy flows in two dimensions. \textit{SIAM J. Math. Anal.} \textbf{49} (2017), 3427--3475.

\bibitem{CDX1} J. Cheng, L. Du, and W. Xiang,
Incompressible jet flows in a de Laval nozzle with smooth detachment.
\textit{Arch. Ration. Mech. Anal.}  \textbf{232} (2019), 1031--1072.

\bibitem{CDX2} J. Cheng, L. Du, and W. Xiang,
Compressible subsonic jet flows issuing from a nozzle of arbitrary cross-section.
\textit{J. Diff. Eqs.} \textbf{266} (2019), 5318--5359.

\bibitem{CoF} R. Courant and K. O. Friedrichs,
\textit{Supersonic Flow and Shock Waves}. Springer-Verlag: New York, 1948.

\bibitem{Dafermos} C. Dafermos,
\textit{Hyperbolic Conservation Laws in Continuum Physics},
4th Ed., Springer-Verlag, Berlin, 2016.

\bibitem{EL}
V. Elling and T.-P. Liu,
Supersonic flow onto a solid wedge,
\textit{Comm. Pure Appl. Math.} \textbf{61} (2008), 1347--1448.

\bibitem{Fang} B. Fang, Stability of transonic shocks for the full
Euler system in supersonic flow past a wedge. \textit{Math. Methods
Appl. Sci.} \textbf{29} (2006), 1--26.

\bibitem{FX} B. Fang and W. Xiang,
The uniqueness of transonic shocks in supersonic flow past a 2-D wedge.
\textit{J. Math. Anal. Appl.} \textbf{437} (2016), 194--213.

\bibitem{GH}
\newblock  D. Gilbarg and L. H\"ormander,
\newblock Intermediate Schauder estimates,
\newblock {\it Arch. Ration. Mech. Anal.}  \textbf{74} (1980),
297--318.

\bibitem{gt}
\newblock D. Gilbarg and N. Trudinger,
\newblock
{\it Elliptic Partial Differential Equations of Second Order}. \newblock 2nd
Ed., Springer-Verlag: Berlin, 1983.

\bibitem{Gu} C.-H. Gu, A method for solving the supersonic flow past
a curved wedge. \textit{Fudan J. $($Nature Sci.$)$}, \textbf{7} (1962),
11--14.

\bibitem{HKWX}
F. Huang, J. Kuang, D. Wang, and W. Xiang,
Stability of supersonic contact discontinuity for two-dimensional steady compressible Euler flows in a finite nozzle.
\textit{J. Diff. Eqs.} \textbf{266} (2019), 4337--4376.

\bibitem{LXY}
L. Li, G. Xu, and H.~C. Yin,
On the instability problem of a 3-D transonic oblique shock wave.
\textit{Adv. Math.} \textbf{282} (2015), 443--515.

\bibitem{Lieberman1}
\newblock G. Lieberman,
\newblock Mixed boundary value problems for elliptic and parabolic differential equations
of second order.
\newblock {\it J. Math. Anal. Appl.}  \textbf{113} (1986), 422--440.

\bibitem{Lieberman2}
\newblock G. M. Lieberman,
\newblock Oblique derivative problems in Lipschitz domains II. Discontinuous boundary data.
\newblock {\it J. Reine Angew. Math.}  \textbf{389} (1988), 1--21.

\bibitem{Liu}
\newblock T.-P. Liu
\newblock Multi-dimensional gas flow: some historical perespectives.
\newblock {\it Bull. Inst. Math. Acad. Sinica $($New Series$)$}, \textbf{6} (2011), 269--291.

\bibitem{Prandtl}
L. Prandtl,  Allgemeine \"{U}berlegungen \"{u}ber die Str\"{o}mung zusammendr\"{u}ckbarer Fl\"{u}ssigkeiten.
{\em Zeitschrift f\"{u}r angewandte Mathematik und Mechanik}, \textbf{16} (1936), 129--142.

\bibitem{QW}
A. Qu and W. Xiang,
Three-dimensional steady supersonic Euler flow past a concave cornered wedge with lower pressure at the downstream.
\textit{Arch. Ration. Mech. Anal.} \textbf{228} (2018), 431--476.

\bibitem{Schaeffer} D. Schaeffer, Supersonic flow past a nearly
straight wedge. \textit{ Duke Math. J.} \textbf{43} (1976), 637--670.

\bibitem{Serre}
D. Serre,  von Neumann's comments about existence and uniqueness for the initial-boundary
value problem in gas dynamics.
{\em Bull. Amer. Math. Soc. $($N.S.$)$}, \textbf{47} (2010), 139--144.

\bibitem{Neumann}
John von Neumann, Discussion on the existence and uniqueness 	
or multiplicity of solutions of the aerodynamical equations. 	
{\em Bull. Amer. Math. Soc.} \textbf{47} (2010), 145--154. 	

\bibitem{XZZ}
W. Xiang, Y. Zhang, and Q. Zhao,
Two-dimensional steady supersonic exothermically reacting Euler flows with strong contact discontinuity over Lipschitz wall.
\textit{Interfaces Free Bound.} \textbf{20} (2018), 437--481.

\bibitem{YinZhou2009JDE}
H. Yin and C. Zhou, On global transonic shocks for the steady supersonic Euler
flows past sharp 2-D wedges. 	
{\em J. Diff. Equ.} \textbf{246} (2009), 4466--4496. 	


\bibitem{Zh1} Y. Zhang, Global existence of steady supersonic potential
flow past a curved wedge with piecewise smooth boundary.
\textit{SIAM J. Math. Anal.} \textbf{ 31} (1999), 166--183.

\bibitem{Zh2} Y. Zhang, Steady supersonic flow past an almost straight
wedge with large vertex angle. \textit{J. Diff. Eqs.} \textbf{192}
(2003), 1--46.
\end{thebibliography}


\end{document}